\documentclass[12pt]{amsart}

\usepackage{hyperref}
\usepackage{braket}

\usepackage{amsmath,amstext,amssymb,amsopn,amsthm}
\usepackage{url,verbatim}
\usepackage{mathtools}
\usepackage{enumerate}

\usepackage{color,graphicx}

\usepackage[margin=30mm]{geometry}
\usepackage{eucal,mathrsfs,dsfont}

\allowdisplaybreaks

\newtheorem{theorem}{Theorem}[section]

\newtheorem{lemma}[theorem]{Lemma}
\newtheorem{proposition}[theorem]{Proposition}

\theoremstyle{definition}

\newtheorem{definition}[theorem]{Definition}

\newtheorem{remark}[theorem]{Remark}
\newtheorem{example}[theorem]{Example}

\numberwithin{equation}{section}

\newcommand{\calG}{\mathcal{G}}
\newcommand{\calA}{\mathcal{A}}

\newcommand{\calE}{\mathcal{E}}
\newcommand{\calB}{\mathcal{B}}
\newcommand{\calV}{\mathcal{V}}
\newcommand{\calH}{\mathcal{H}}
\newcommand{\calN}{\mathcal{N}}

\newcommand{\E}{\operatorname{\mathds{E}}} 
\renewcommand{\P}{\operatorname{\mathds{P}}} 

\newcommand{\V}{{\mathds{V}}}
\newcommand{\Z}{{\mathbb Z}}

\newcommand{\prt}{\partial}
\newcommand{\lra}{\leftrightarrow}
\newcommand{\bfG}{\mathbf{G}}

\newcommand{\ol}{\overline}

\newcommand{\wt}{\widetilde}

\DeclareMathOperator{\dist}{dist}

\def\bone{{\bf 1}}
\def\n{{\bf n}}

\renewcommand{\comment}[1]{}
\newcommand{\eq}{\begin{equation}}
\newcommand{\en}{\end{equation}}

\newcommand{\tree}{\mathcal{G}}

\title{Twin peaks}
\author{Krzysztof Burdzy and Soumik Pal}

\address{ Department of Mathematics, Box 354350, University of Washington, Seattle, WA 98195}
\email{burdzy@uw.edu}

\email{soumik@uw.edu}

\thanks{Research supported in part by NSF Grants DMS-1206276 for Burdzy and DMS-1308340 for Pal. }

\begin{document}

\begin{abstract}
We study random labelings of graphs conditioned on a small number (typically one or two) peaks, i.e., local maxima. We show that the boundaries of level sets of a random labeling of a square with a single peak have dimension 2, in a suitable asymptotic sense. The gradient line of a random labeling of a long ladder graph conditioned on a single peak consists mostly of straight line segments. We show that for some tree-graphs, if a random labeling is conditioned on exactly two peaks then the peaks can be very close to each other. We also study random labelings of regular trees conditioned on having exactly two peaks. Our results suggest that the top peak is likely to be at the root and the second peak is equally likely, more or less, to be any vertex not adjacent to the root.

\end{abstract}

\maketitle

\section{Introduction}\label{intro}

This paper is the first part of a project motivated by our desire to understand random labelings of graphs conditioned on having a small number of peaks (local maxima). 
The roots of this project lie in the paper \cite{BBS}, where some new results on peaks of random permutations were proved. These were later applied in \cite{BBPS}.

To indicate the direction of our thinking, we now state informally an open problem
that was the target of our initial investigation. 
Consider a discrete $2$-dimensional torus with $N$ vertices and let $L$ be a random labeling of the 
vertices of the torus with integers $1,2, \dots, N$. If we condition   $L$ on having exactly two peaks and we call their locations $K_1$ and $K_2$, is it true that $\dist(K_1,K_2)/\sqrt{N} \to 0$ in probability as $N\to \infty$?
We conjecture that the answer is ``no.''

We are grateful to Nicolas Curien for pointing out to us the following related articles. Local maxima of labelings are analyzed in \cite{AmBudd}. The labels (ages) of vertices in a random tree generated according to the Barab\'asi and Albert
``preferential attachment model'' (see \cite[Chap.~8]{VDH}), conditioned on a given tree structure, have the distribution of a random labeling conditioned on having a single peak at the root. These research directions are significantly different from ours.

The paper is organized as follows. 
It starts with Section 
\ref{sec:model} presenting a comparison between a $1$-dimensional model and a $2$-dimensional model, and another comparison of two $2$-dimensional models, as a further motivation for our study. 
Section
\ref{prel} contains some preliminaries (notation and basic formulas).  In Sections \ref{rough} and \ref{ladder} we will prove two results on random labelings of tori inspired by  simulations. Section \ref{tree} will present results on random labelings of trees. One can prove more precise results for random labelings of \emph{regular} trees---these will be given in Section \ref{regtrees}. 

\section{Models and simulations}\label{sec:model}

We will define a random labeling of $N$ vertices of $\Z^d$ and a random labeling of $d$-dimensional torus with $N$ vertices. The two algorithms  generate similarly looking random labelings in dimension $d=1$, but computer simulations and some rigorous results (to be presented later in this paper) show that the two algorithms generate strikingly different random labelings in $\Z^2$.

(i) Fix an integer $N>0$. We will define inductively a random  function
$\wt L: \Z^d \to \{0,1,2,\dots, N\}$ such that the origin is mapped onto $N$ and $\wt L^{-1} (j)$ is a singleton for $1 \leq j \leq N$. Let $\wt C_k$  be the inverse image of $\{N, N-1, \dots, N-k\}$ by $\wt L$  for $0 \leq k \leq N-1$. Given $\wt C_k$ for some $0\leq k \leq N-2$, the vertex $\wt L^{-1} (N-k-1)$ is distributed uniformly among all vertices in $\Z^d \setminus \wt C_k$ which  have a neighbor in $\wt C_k$. 
We let $\wt L(x)=0$ for all $x\notin \wt C_{N-1}$.
It is easy to see that  the only peak (strict local maximum) of $\wt L$ is at $\wt L^{-1} (N)$.

(ii) Suppose that $N_1 >0$,  $N = N_1^d$ and let $\calG$ be the $d$-dimensional torus with side length $N_1$ and vertex set $\calV$. Let $ L: \calV \to \{1,2,\dots, N\}$ 
be a random (uniformly chosen) bijection conditioned on  $N$ being the only peak.
Let $C_k$ be the inverse image of $\{N, N-1, \dots, N-k\}$ by $L$ for $0 \leq k \leq N-1$. 

Note that  both $\wt C_k$ and $C_k$ are connected for every $0 \leq k \leq N-1$. We also have $\wt C_{k} \subset \wt C_{k+1}$ and $C_{k} \subset C_{k+1}$ for all $k$.

Methods applied in \cite{BBS} and other standard techniques show that
in dimension $d=1$, for large $N$ and $k$ comparable with $N$, both $\wt C_k$ and $C_k$ are intervals centered, more or less, at the vertex labeled $N$, with high probability.

In dimension $d=2$, the cluster growth process $\{\wt C_k, k\geq 1\}$ is the Eden model, \cite{Eden}. See Fig.~\ref{fig3}. Convincing heuristic arguments (see Remark \ref{m30.1}) show that the boundary of $\wt C_k$ has the size of the order $k^{1/2}$.  

\begin{figure} \includegraphics[width=0.9\linewidth]{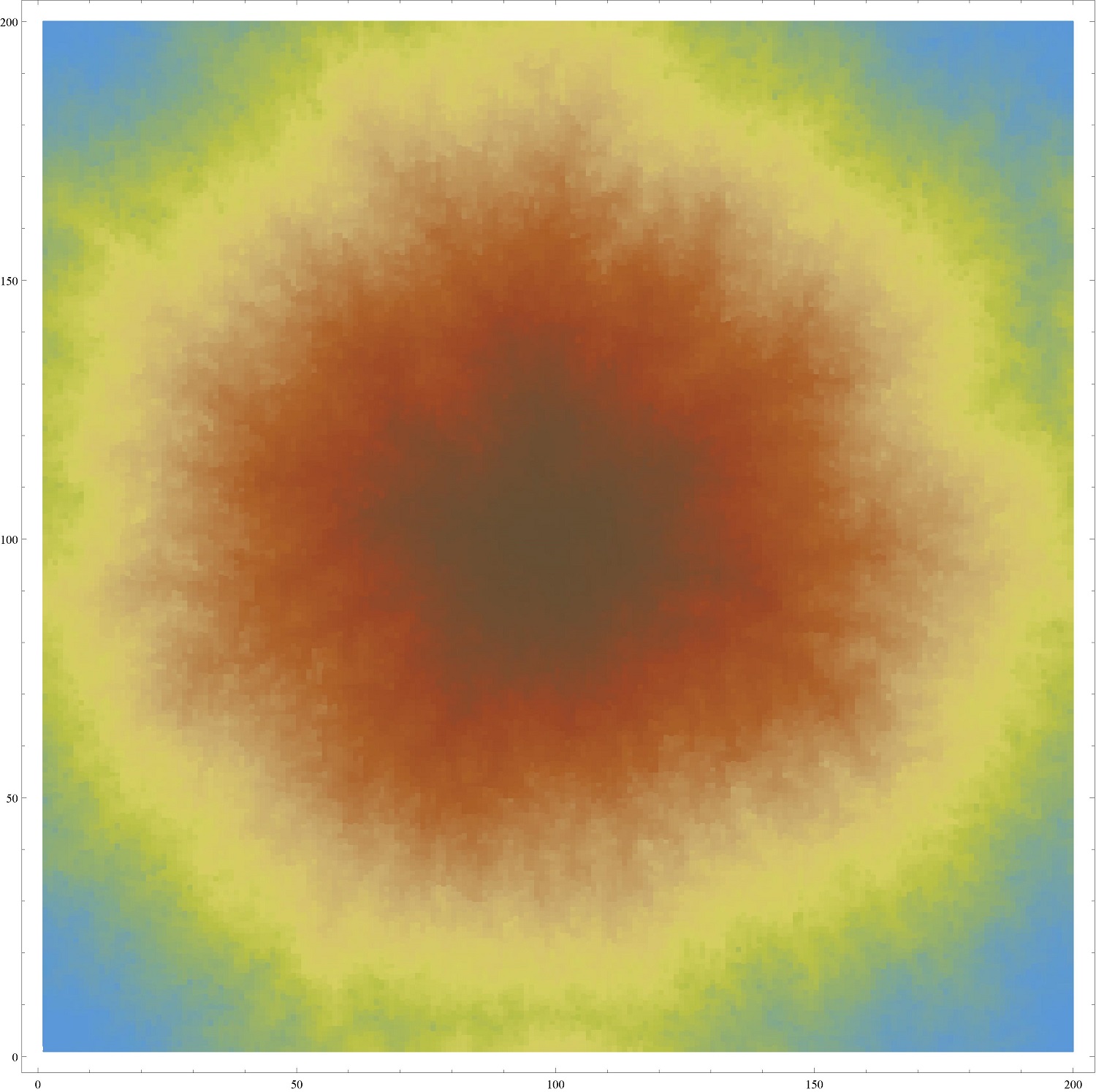}
\caption{
Eden model on the discrete torus with edge length 200. The original cluster was a single point at $(100,100)$. The order of point attachment to the cluster is indicated by the warmth of the color.
}
\label{fig3}
\end{figure}

The contrast between Figs.~\ref{fig3} and \ref{fig2} could not be more dramatic. Fig.~\ref{fig2} is a simulation of the process $\{ C_k, k\geq 1\}$ in two dimensions. We will show that the boundaries of the sets $C_k$ have dimension 2, in an asymptotic sense.

\begin{figure} \includegraphics[width=0.9\linewidth]{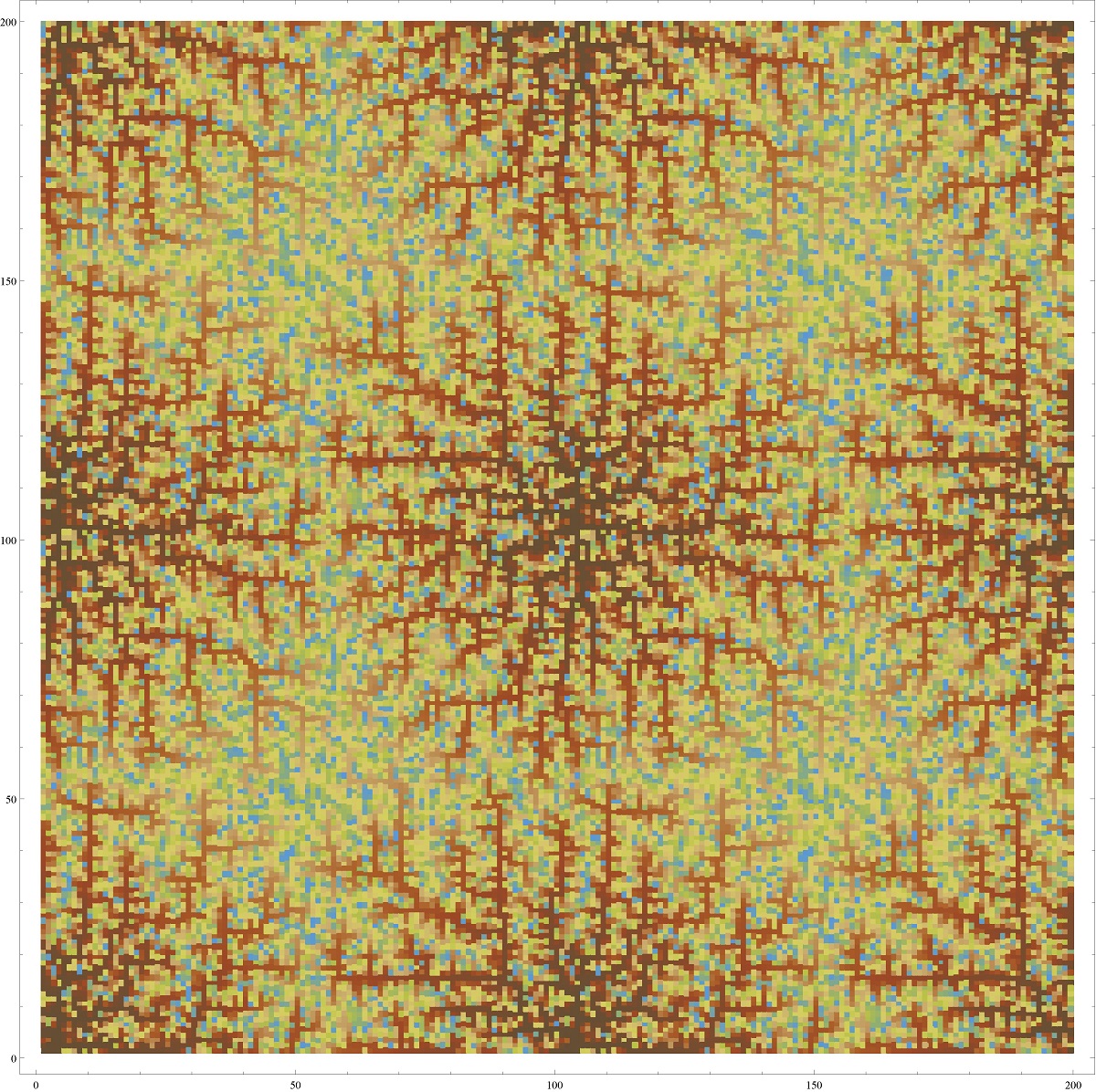}
\caption{
Random labeling of the discrete torus with edge length 100. The labeling is conditioned on having exactly one peak. Level lines are indicated by the warmth of the color. The figure consists of four copies of the same simulation, placed next to each other, to help see the main geometric features.
}
\label{fig2}
\end{figure}

The following elementary example shows that 
models (i) and (ii) are not equivalent for general  graphs. In other words, the sets $C_k$ associated with the (uniformly chosen) random labeling with a single peak do not grow by attaching vertices uniformly on the boundary.

\begin{example}\label{j15.2}

Consider a graph with vertices $\{a,b,c,d\}$ and edges $\{\ol{ab},\ol{bc},\ol{cd}\}$ (a linear graph, with $a,b,c$ and $d$ arranged along a line in this order). Let $L$ be the uniformly random labeling conditioned on having a single peak at $b$. 
Let $R = L^{-1}$ be the inverse function of $L$.
The following three realizations of $L$ are possible and equally likely.
\begin{align*}
L(a) = 3, L(b) = 4, L(c) = 2, L(d) = 1;\\
L(a) = 2, L(b) = 4, L(c) = 3, L(d) = 1;\\
L(a) = 1, L(b) = 4, L(c) = 3, L(d) = 2.
\end{align*}
We see that $R(3)$ can be equal to $a$ or $c$ but 
\begin{align*}
1/3 =
\P(R(3) = a \mid R(4) = b) \ne \P(R(3) = c \mid R(4) = b) =2/3.
\end{align*}
In other words, if $R(3)$ is chosen uniformly from all vertices adjacent to $R(4)$ and this is followed by any choice of $R(2)$ and $R(1)$ consistent with the condition that $R(4)=b$ is the only peak of $L$ then the resulting labeling is not uniformly random (conditional on  $\{R(4)=b\}$).

\end{example}

\section{Preliminaries}\label{prel}

For an integer $N>0$, let $[N] =\{1,2,\dots, N\}$.
By an abuse of notation, we will use $\lvert \cdot \rvert$  to denote the absolute value of a real number and cardinality of a finite set.

We will usually denote graphs by $\calG$, their vertex sets by $\calV$ and edge sets by $\calE$.
We will indicate adjacency of vertices $x$ and $y$ by $x\lra y$.

Suppose that  $1<|\calV|= N<\infty$.  We will call a function $L: \calV \to [N]$ a labeling if it is a bijection. 
 
We will say that a vertex $x\in \calV$ is a peak (of $L$) if and only if it is a local maximum of $L$, i.e., $L(x) > L(y)$ for all 
$y\lra x$. The event that the number of peaks of a random labeling  of graph $\calG$ is equal to $k$ will be denoted $\calN(\calG,k)$.

We will use $\P$ to denote the distribution of a random (uniform, unconditioned) labeling. The symbol $\P_1$ ($\P_2$) will stand for the distribution $\P$ conditioned on existence of exactly one peak (exactly two peaks). If an argument involves multiple graphs, we may indicate the graph in the notation by writing, for example, $\P_{\calG}$ or $ \P_{1, \calG}$.

Recall the Stirling formula for $n\geq 1$, 
\begin{align}\label{f15.1}
\log n! = n \log n - n + O(\log n).
\end{align}
This obviously implies that,
\begin{align}\label{j20.1}
\log n! \leq n \log n  + O(\log n).
\end{align}
A more accurate version of the Stirling formula says that, for $n\geq 1$,
\begin{align}\label{ap27.1}
(n+1/2) \log n - n + \log \sqrt{2 \pi} \leq
\log n! \leq (n+1/2) \log n - n + 1.
\end{align}

For every $k=0,1,\dots, n$, we have
$\binom{n}{k} \le \sum_{j=0}^n \binom{n}{j}=2^n$,
so
\begin{align}\label{ap28.1}
\log \binom{n}{k} &\leq  n \log 2.
\end{align}

\section{Roughness of level boundaries}\label{rough}

Recall notation from Section \ref{sec:model}. We will show that  boundaries of  clusters $C_k$ (``level sets'') for random labelings of two-dimensional tori, conditioned on having a single peak, have dimension 2, in a suitable asymptotic sense.

\begin{remark}\label{m30.1}

Recall the Eden model, a sequence of randomly growing clusters $\wt C_k$, from Section \ref{sec:model}.
The  model was introduced in \cite{Eden}. The site version of the Eden model
is identical to the site version of the ``first passage percolation'' model
introduced in \cite{HW}. A shape theorem was proved in \cite{CD} for the
edge version of the model; it says that when the size of the cluster
goes to infinity, the rescaled cluster converges to a convex set $B$.
For a connection between the Eden model and the KPZ  equation, see, e.g., \cite{Quast}.

Theorem 3.1 of \cite{Alex} (see also a stronger version in \cite[Prop.~3.1]{DK}) implies that with probability 1, for all large $n$, the boundary of the set $\wt C_n$ lies in the ``annulus'' of width $c _1n^{1/4}\log n$ containing $n^{1/2}\prt B$ or, more precisely, a neighborhood of width $c_1 n^{1/4}\log n$ of the boundary of $n^{1/2}B$. This implies that the cardinality of the boundary of $\wt C_n$ is bounded by the volume of the annulus, that is,
$c_2 n^{3/4} \log n$. 

The following conjecture was communicated to us by Michael Damron.  For some constant $c_3$, with probability one, there exist infinitely many $n$ such that the boundary of $\wt C_n$ contains at most $c_3 n^{1/2}$ vertices. 

Even the weaker upper bound, $c_2 n^{3/4} \log n$, on the size of cluster boundaries in the Eden model
and Theorem \ref{f12.3} show that the clusters $C_k$ in the random labeling model have much different character from those in the Eden model. The comparison of Figs.~\ref{fig3} and \ref{fig2} seems to support this claim.
\end{remark}

Recall that $\calG_{n,n}$ denotes the Cartesian product 
of path (linear) graphs, each with $n$ vertices.
In other words, $\calG_{n,n}$ is a discrete square with edge length $n$.
The set of vertices of $\calG_{n,n}$ is denoted $\calV_{n,n}$.

Recall clusters $C_k$ defined in Section \ref{sec:model}. 
Let $\prt C_k$ denote the set of vertices outside $C_k$ that are adjacent to a vertex in $C_k$.

\begin{theorem}\label{f12.3}
Consider a random labeling of $\calG_{n,n}$ conditioned on having a single peak.
 Let $f(n) = (\log n )^{-1/2}\log\log n$. Let $F$ be the event that the number of $k\in [n^2]$ such that $|\prt C_k| \leq  n^{2- f(n)}$ is greater than or equal to  $f(n)n^2$. For large $n$,
\begin{align*}
\P_1(F) \leq \exp(- (1/2)n^2 (\log\log n)^2).
\end{align*}
\end{theorem}

\begin{proof}

\noindent{\it Step 1}.
In this step, we will prove that
for
any $1\leq m_1, m_2 \leq n$, the number of labelings
of $\calG_{n,n}$ which have only  one peak at $(m_1, m_2)$ is greater than
\begin{align}\label{f12.1}
\exp\left(2 n^2 \log n + o(n^2 \log n)\right).
\end{align}

Let $\ell= \ell(n) = \lfloor \log n \rfloor$.
Let $\calH_n$ be the graph whose vertex set
is a subset of $\calV_{n,n}$ and contains all
points  with coordinates
$(k_1, k_2)$ satisfying at least one of the following four conditions: (i) $k_1 = m_1$, (ii) $k_2 = m_2$,
(iii) $k_1 = j \ell$ for some integer $j$,
(iv) $k_2 = j \ell$ for some integer $j$.
Two vertices of $\calH_n$ are connected by an edge if and only
if they are adjacent in $\calG_{n,n}$.
The graph $\calH_n$ has the shape of a lattice with the edge length
$\ell$ and a pair of extra lines passing through $(m_1, m_2)$.

Let $N$ be the number of vertices in $\calH_n$. We have
$N \leq 3 n^2/ \log n$ for large $n$. Let $\calV^\calH_n$ denote
the set of vertices of $\calH_n$.

We will define a family of deterministic labelings $L$ of $\calG_{n,n}$.
We will first label $\calV^\calH_n$.
It is easy to see that one can label
$\calV^\calH_n$ in such a way that $L(m_1,m_2) = n^2$,
$(m_1,m_2)$ is the only peak of $L$, and 
$L(\calV^\calH_n) = \{n^2, n^2-1, \dots, n^2 - N+1\}$.
Fix any such labeling of $\calV^\calH_n$ and consider
labelings $L$ of $\calV_{n,n}$ which agree with this labeling
on $\calV^\calH_n$.

Consider labelings with no restriction on peaks outside
$\calV^\calH_n$. There are $n^2-N$ numbers in $[n^2]$ not used for labeling $\calV^\calH_n$ so there
are $(n^2-N)!$  labelings of $\calG_{n,n}$ which extend a given labeling of $\calV^\calH_n$. Let this family be called
$\calA$. 

Consider a maximal horizontal line segment $\Gamma \subset \calV_{n,n} \setminus \calV^\calH_n$.
The length of $\Gamma$ is at most $\ell -1$. At least one of the permutations of $L(\Gamma)$ is monotone  with the largest number adjacent to $\calV^\calH_n$.
So there are at least $(n^2-N)! /(\ell-1)!$ labelings in $\calA$ with no peaks on $\Gamma$. Let $\Gamma_1$ be  a maximal horizontal line segment in $ \calV_{n,n} \setminus \calV^\calH_n$, disjoint from $\Gamma$. We can apply the same argument using this time permutations of  $L(\Gamma_1)$ to see that the number of labelings in $\calA$  with no peaks on $\Gamma \cup \Gamma_1$ is greater than or equal to $(n^2-N)! ((\ell-1)!)^{-2}$. The set $ \calV_{n,n} \setminus \calV^\calH_n$ can be partitioned into at most $n(n+1)/\ell \leq  (n^2 /\log n) (1+1/n)(1+ 2/\log n)$ disjoint maximal horizontal line segments of length at most $\ell$, for large $n$. So 
the number of labelings in $\calA$  with no peaks  is greater than or equal to $(n^2-N)! ((\ell-1)!)^{-(n^2 /\log n) (1+1/n)(1+ 2/\log n)}$.
For large $n$, the logarithm of this number is greater than or equal to, using \eqref{ap27.1}, 
\begin{align*}
&(n^2 - N +1/2) \log (n^2 -N) - (n^2 - N) + c_1 \\
&\quad - (n^2 /\log n) (1+1/n)(1+ 2/\log n)(\ell-1+1/2) \log (\ell -1)\\
& \geq (n^2 - 3 n^2/ \log n +1/2) \log (n^2 -3 n^2/ \log n) - n^2  \\
&\quad - (n^2 /\log n) (1+1/n)(1+ 2/\log n)\log n \log \log n \\
&\geq 2n^2(1 - 3/ \log n) \log n 
+n^2(1 - 3/ \log n)\log(1 - 3/ \log n) -n^2 \\
&\quad - n^2 (1+1/n)(1+ 2/\log n) \log \log n\\
& \geq 2 n^2 \log n + O(n^2 \log\log n).
\end{align*}
This proves \eqref{f12.1}.

\medskip

\noindent{\it Step 2}.
Suppose that $B_1, B_2, \dots, B_{n^2}$ is a  sequence of sets in $\calV_{n,n}$, starting with a singleton $B_1$ and such that $|B_k|=|B_{k-1}|+1$ for all $k$. 
Let $\prt B_k$ denote the set of vertices outside $B_k$ that are adjacent to a vertex in $B_k$. We also require that $B_{k+1} \setminus B_k \subset \prt B_k$ for all $k$.
Recall  that $f(n) = (\log n )^{-1/2}\log\log n$. Let $M$ be the number
of  sequences $(B_k)_{k\in[n^2]}$ with the property  that the number $R$ of $k\in [n^2]$ such that $|\prt B_k| \leq  n^{2- f(n)}$ is greater than or equal to  $f(n)n^2$.

The number of possible values of $R$ is bounded by $n^2$.
Given $R\geq f(n)n^2$, the number of possible sequences of $k$'s such that $|\prt B_k| \leq  n^{2- f(n)}$ is bounded by $\binom{n^2}{R}$. 
Given a specific set $\Lambda$ of $k$'s in $[n^2]$ with cardinality greater than or equal to  $f(n)n^2$, the number
of  sequences $(B_k)_{k\in[n^2]}$ with the property  that  $|\prt B_k| \leq  n^{2- f(n)}$ for all $k \in \Lambda$ is less than or equal to  
\begin{align*}
 (n^{2-f(n)})^{f(n)n^2} (n^2)^{(1-f(n))n^2+1},
\end{align*}
so
\begin{align*}
M \leq n^2 \binom{n^2}{R} (n^{2-f(n)})^{f(n)n^2} (n^2)^{(1-f(n))n^2+1}.
\end{align*}
By \eqref{ap28.1}, for some $c$ and all $n$ and $R$, $\log\binom{n^2}{R} \leq c n^2$. Hence,
\begin{align*}
\log M &\leq 
2 \log n + c n^2 
 +f(n)n^2 (2-f(n)) \log n + 2 ((1-f(n))n^2+1) \log n \\
&= c n^2 - n^2 f^2(n) \log n + 2 n^2 \log n + 4 \log n\\
&=c n^2  - n^2 (\log n )^{-1}(\log\log n)^2 \log n + 2 n^2 \log n + 4 \log n\\
& =c n^2 - n^2 (\log\log n)^2  + 2 n^2 \log n + 4 \log n.
\end{align*}
The probability of $F$ is less than or equal to $M$ divided by the number in \eqref{f12.1}, so for some $c$ this probability is bounded above by
\begin{align*}
&\frac{\exp(c n^2 - n^2 (\log\log n)^2  + 2 n^2 \log n + 4 \log n)}
{\exp(2 n^2 \log n - cn^2 \log\log n)}\\
&\quad= \exp(- n^2 (\log\log n)^2 +cn^2  +4 \log n + cn^2 \log\log n)\\
&\quad\leq \exp(- (1/2)n^2 (\log\log n)^2), 
\end{align*}
for all large $n$.
\end{proof}

\section{Ladder graphs}\label{ladder}

The results of a simulation presented in Fig.~\ref{fig2} suggest that  gradient lines for the random labeling contain long straight stretches. We cannot prove this feature of random labelings of tori but we will prove a similar result for ``ladder'' graphs. Theorem \ref{f14.1} will show that the longest of all gradient lines of a random labeling with a single peak of a ladder graph mostly consists of straight stretches, with the exception of a logarithmically small percentage of steps.

A generalized ladder graph $\calG_{m,n}$ is the Cartesian product
of path (linear) graphs with $m$ and $n$ vertices. 
In other words, $\calG_{m,n}$ is a discrete rectangle with edge lengths $m$ and $n$.
We will consider $m$  a fixed parameter and we will state an asymptotic theorem when $n$ goes to infinity.

We will identify the set of vertices $\calV_{m,n}$ of the graph $\calG_{m,n}$ with points of $\Z^2$, so that we can refer to them using Cartesian coordinates 
$(j,k)$, $1\leq j \leq n$, $1\leq k \leq m$.
Let $\calB$ be the family of all (deterministic) labelings of $\calG_{m,n}$ with only one peak at $(1,1)$. 
Let $\Lambda$ be the set of all ``continuous'' 
non-self intersecting paths taking values in $\calV_{m,n}$, i.e., $\lambda\in \Lambda$ if and only if there is some $k$ such that $\lambda: [k] \to \calV_{m,n}$, $\lambda(j) \lra \lambda(j+1)$
for $j\in [k-1]$, and $\lambda(i) \ne \lambda(j)$ for $i\ne j$. We will write $|\lambda|=k$. 

We will argue that for any $L\in\calB$ there exists $\lambda_L = (\lambda_L^1, \lambda_L^2) \in \Lambda$ such that 
 $\lambda_L(1) = (1,1)$, $\lambda_L^1(|\lambda_L|) = n$, and $j\to L(\lambda_L(j))$ is a decreasing function.
Fix any ordering of $\calV_{m,n}$.  We will construct a tree. Let $v_1=(1,1)$ be the root of the tree. 
Suppose that vertices $\{v_1, v_2, \dots,v_j\}$ of the tree have been chosen. Consider the following cases. 
(a) The set $\calV_{m,n}\setminus \{v_1, v_2, \dots,v_j\}$ contains a vertex $w$ such that $w$ is adjacent to a vertex in  $\{v_1, v_2, \dots,v_j\}$, say, $v_i$, and $L(w) < L(v_i)$. Then we take the largest $w$ with this property and add this $w$ to $\{v_1, v_2, \dots,v_j\}$ as a new element $v_{j+1}$. We also choose the largest $v_i \in \{v_1, v_2, \dots,v_j\}$ with the property that $L(v_{j+1}) < L(v_i)$ and add an edge to the tree between $v_{j+1}$ and $v_i$.
(b) Suppose that the set $\calV_{m,n}\setminus \{v_1, v_2, \dots,v_j\}$ does not contain a vertex $w$ such that $w$ is adjacent to a vertex in  $\{v_1, v_2, \dots,v_j\}$ and has the property that $L(w) < L(v_i)$ for some $v_i$.
If $\{v_1, v_2, \dots,v_j\}$ contains a vertex  of the form $(n,k)$ then the unique path within the tree from $v_1=(1,1)$ to $(n,k)$ satisfies all conditions that $\lambda _L$ is supposed to satisfy. In the opposite case,  $L$ restricted to the set $\calV_{m,n}\setminus \{v_1, v_2, \dots,v_j\}$, must attain the maximum, say, at $x$. If $x$ is adjacent to any vertex $y \in \{v_1, v_2, \dots,v_j\}$ then $L(x) > L(y)$ for every such $y$ because of the assumption made at the beginning of case (b). This implies that $x$ is a peak of $L$ in the whole graph $\calG_{m,n}$. Since this contradicts the assumption that $L$ has only one peak at $(1,1)$, the proof of existence of $\lambda_L$ is complete. 

If there is more than one labeling $\lambda_L$ satisfying the properties stated above then we 
let $\lambda_L$ be the labeling that, in addition, is minimal in some fixed arbitrary ordering of $\Lambda$. 

\begin{theorem}\label{f14.1}
Let $L$ be a random labeling of $\calV_{m,n}$, chosen uniformly from $\calB$ and let $\P_\calB$ denote the corresponding probability. For every $m \geq 3$
there exists  $n_1$ such  that for $n\geq n_1$,
\begin{align*}
\P_\calB\left(|\lambda_L| \geq n \left( 1 + \frac{\log\log n}{\log n}\right) \right) \leq \exp(-(1/2)n \log\log n).
\end{align*}
\end{theorem}

\begin{proof}

First we will estimate the total number of deterministic labelings of $\calV_{m,n}$ with a single peak at $(1,1)$. 

Let $\calA$ be the family of all deterministic labelings $L$ of $\calV_{m,n}$ such that 
\begin{align*}
L(1,1) = nm, \ L(2,1) = nm-1, \dots, L(n,1)= nm - n +1.
\end{align*}
Note that labelings in the family $\calA$ may have multiple peaks. There are $(n(m-1))!$ labelings in $\calA$.

Let $\Gamma_k= \{(k,2), (k,3), \dots, (k,m)\}$ for $k\in[n]$. There exists a unique permutation of $L(k,2), L(k,3), \dots, L(k,m)$ which is decreasing so the number of labelings in $\calA$ which do not have peaks on $\Gamma_1$ is at least 
$(n(m-1))!/ (m-1)!$. For the same reason, the number of labelings in $\calA$ which do not have peaks on $\Gamma_1\cup \Gamma_2$ is at least 
$(n(m-1))!/ ((m-1)!)^2$. Extending the argument to all $\Gamma_k$'s, we conclude that the number of labelings in $\calA$ which do not have a peak 
on any $\Gamma_k$
 is at least 
$(n(m-1))!/ ((m-1)!)^n$. This implies that the number of labelings with a single peak at $(1,1)$ is bounded below by
\begin{align}\label{f14.2}
(n(m-1))!/ ((m-1)!)^n.
\end{align}

Let $f(n) = \log\log n/\log n$.
Consider any integer $j$ such that $n (1+ f(n))\leq j\leq  nm$.
We will count the number of paths $\lambda\in\Lambda$ such that $|\lambda|=j$
and $\lambda = \lambda_L$ for some $L\in\calB$.
We must have $L(1,1) = nm$ and $\lambda(1) = (1,1)$.
There are at most $\binom{nm}{j-1}$ choices for the values of $L(\lambda(k))$, $2\leq k\leq j$. 
Once these values are chosen, we arrange them along a ``continuous'' path
by choosing inductively a vertex $\lambda(k+1)$, given $\lambda(1), \dots, \lambda(k)$. 
Going from the top value monotonically to the lowest value, there are at most 3 choices at every step so  the number of $\lambda$ such that  $|\lambda| = j$ and $\lambda = \lambda_L$ for some $L\in\calB$ is bounded by $\binom{nm}{j-1}3^{j-1}$.

If $|\lambda| = j$ then there are $nm-j$ vertices outside the range of $\lambda$. Hence, the number of labelings $L\in\calB$ such that $\lambda_L$ is equal to a fixed $\lambda$ with $|\lambda|=j$ 
and the values of $L(\lambda(k))$ are fixed for $1\leq k \leq j$,
is bounded by $(nm - j)!$. Let $M(m,n,j)$ be the number of labelings $L\in\calB$ such that  $|\lambda_L|=j$.
We have
\begin{align*}
M(m,n,j)\leq
\binom{nm}{j-1}3^{j-1} (nm - j)! 
\leq \binom{nm}{j-1}3^{j-1} (nm - j+1)! 
= \frac{(nm)!}{(j-1)!} 3^{j-1}.
\end{align*}
For $n\geq 4$, the expression on the right hand side, considered to be a function of $j$ on the interval $n (1+ f(n))\leq j\leq  nm$, is maximized by $j=\lceil n (1+ f(n))\rceil$. 
Let $n_1 =\lceil n (1+ f(n))\rceil -1$. 
We obtain the following bound, 
\begin{align*}
M(m,n,j)\leq (nm)! 3^{n_1} / n_1 !.
\end{align*}

Next we sum over $j$ such that $n (1+ f(n))\leq j\leq  nm$. The number of labelings $L\in\calB$ such that  $|\lambda_L|\geq n (1+ f(n))$ is bounded by 
\begin{align}\label{f14.3}
n(m-f(n))(nm)! 3^{n_1} / n_1 !.
\end{align}
The probability that $|\lambda_L|\geq n (1+ f(n))$ for a labeling chosen uniformly from $\calB$ is bounded above by the ratio of the numbers in \eqref{f14.3} and \eqref{f14.2}. We use the Stirling formula \eqref{f15.1}
and \eqref{j20.1} to see that
for all sufficiently large $n$,
\begin{align*}
&\P_\calB(|\lambda_L|\geq n (1+ f(n)))\\
&\quad \leq \frac
{n(m-f(n))(nm)! 3^{n_1} / n_1 !}
{(n(m-1))!/ ((m-1)!)^n}\\
&\quad =
\frac
{n(m-f(n))(nm)! 3^{n _1} ((m-1)!)^n}
{(n(m-1))!n_1 !}\\
&\quad \leq \exp\Big(
\log n + \log (m-f(n)) + nm \log (nm)  + O(\log (nm))
+ n (1+ f(n)) \log 3 \\
& \qquad \qquad
+ n (m-1) \log(m-1)  +   O(n\log(m-1)) \\
& \qquad \qquad
- n(m-1)\log(n(m-1) ) + n (m-1) + O(\log(n(m-1)))\\
& \qquad \qquad
- n (1+ f(n)) \log(n (1+ f(n)) -1) + n (1+ f(n)) + O(\log(n (1+ f(n))))
\Big)\\
&\quad \leq \exp\Big( nm \log (nm) 
- n(m-1)\log n 
- n (1+ f(n)) \log(n (1+ f(n)) -1) + O(n)
\Big)\\
&\quad = \exp\Big( nm \log n + nm \log m
- n(m-1)\log n  \\
& \qquad \qquad
- n  \log n
- n  f(n) \log n
- n  \log(1+ f(n) -1/n)  \\
& \qquad \qquad
- n  f(n) \log(1+ f(n) -1/n) 
+ O(n)
\Big)\\
&\quad = \exp\Big(
- n  f(n) \log n
- n  \log(1+ f(n) -1/n) 
- n  f(n) \log(1+ f(n) -1/n) 
+ O(n)
\Big)\\
&\quad = \exp\Big(
- n  f(n) \log n
+ O(n )
\Big)\\
&\quad = \exp\Big(
- n  (\log\log n/\log n) \log n
+ O(n)
\Big)\\
&\quad \leq\exp( -(1/2) n \log\log n ) .
\end{align*}
\end{proof}

\section{Random labelings of trees}\label{tree}

It often happens that trees are more tractable than general graphs. This is indeed the case when we consider random labelings conditioned on a small number of peaks.  We will start this section with some preliminary results.

\begin{definition}\label{j24.2}
(i) Suppose that a graph $\calG$ is a tree with $N$ vertices. We will say that $x$ is a centroid of $\calG$ if each subtree of $\calG$ which does not contain $x$ has at most $N/2$ vertices.

(ii) 
Suppose that $\calG$ is a tree and $x$ is one of its vertices.
 Define a partial order ``$\leq$'' on the set of vertices $\calV$ by declaring that for $y,z\in \calV$, we have $y\leq z$ if and only if $y$ lies on the unique path joining 
 $x$ and $z$.
If $y\leq z$ then we will say that $y$ is an ancestor of $z$ and $z$ is a descendant of $y$. We will write $D^x_y$ to denote the family of all descendants of $y$, including $y$. We will also use the notation $\n_x(y) = |D^x_y|$. If needed, we will indicate the dependence on the graph by adding a superscript, for example, $\n^\calG_x(y)$.

\end{definition}

\begin{remark}\label{m5.1}
The following results can be found in \cite[Thm.~2.3]{BH}.
(i) A finite  tree graph has at least one and at most two centroids. (ii) If it has two centroids then they are adjacent.

Part (i) cannot be strengthened to say that all trees have only one centroid. A linear graph with $N$ vertices has one or two centroids depending on the parity of $N$.

\end{remark}

\begin{proposition}\label{ap1.1}
Suppose that $L$ is a random labeling of a tree $\calG$ and let $K$ denote the location of the highest label. The function $x\to \P(\calN(\calG,1) \cap \{K=x\})$ attains the maximum at all centroids of $\calG$ and only at the centroids. 
\end{proposition}

\begin{proof}

Consider any $x\in \calV$.
Recall the notation $D^x_y$ and $\n_x(y)$ from Definition \ref{j24.2} (ii).
Let
\begin{align*}
F_v = \left\{ L(v) \geq \max_{z \in D^x_v} L(z)\right\}.
\end{align*}
Note that 
\begin{align*}
\calN(\calG,1) \cap \{K=x\} = \bigcap_{v\in \calV} F_v.
\end{align*}

Let $\calV_0 =\calV^+_0 = \{x\}$, $\calV_k = \{v\in \calV \setminus \calV_{k-1}^+:  \exists z \in \calV^+_{k-1} \text{  such that  } \ v\lra z\}$ and $\calV^+_k = \calV_k\cup \calV^+_{k-1}$ for $k\geq 1$. Let $N=|\calV|$ and note that $\P(F_x) = 1/N= 1/ \n_x(x)$. It is easy to see that
\begin{align*}
\P\left(\bigcap_{v\lra x} F_v\; \middle|\; F_x\right) = 
\frac 1 {\prod_{v\lra x} \n_x(v)}.
\end{align*}
In general, for $k\geq 1$,
\begin{align*}
\P\left(\bigcap_{v\in \calV_k} F_v\; \middle|\; \bigcap_{v\in \calV^+_{k-1}} F_v\right) = 
\frac 1 {\prod_{v\in \calV_k} \n_x(v)}.
\end{align*}
It follows that
\begin{align}\label{m21.1}
\P(\calN(\calG,1) \cap \{K=x\}) &=  \P(F_x)\prod_{k\geq 1}
\P\left(\bigcap_{v\in \calV_k} F_v \; \middle| \;\bigcap_{v\in \calV^+_{k-1}} F_v\right)\\
&= \prod_{k\geq 0 } \frac 1 {\prod_{v\in \calV_k} \n_x(v)}
= \frac 1 {\prod_{z\in \calV} \n_x(z)} .\nonumber
\end{align}

We will now vary $x$. Consider two adjacent vertices $x$ and $y$.
Note that $\n_y(v) = \n_x(v)$ for $v\ne x,y$.
 We have
\begin{align*}
\n_y(x) = \n_x(x) - \n_x(y), \qquad \n_y(y) = N = \n_x(x).
\end{align*}
Therefore,
\begin{align}\label{y24.1}
\frac {\P(\calN(\calG,1) \cap \{K=x\})}{\P(\calN(\calG,1) \cap \{K=y\})}
= \frac {\n_y(x) \n_y(y)} {\n_x(x) \n_x(y)}
= \frac {\n_y(x )} { \n_x(y)}
= \frac { \n_x(x) - \n_x(y) } { \n_x(y)}.
\end{align}
It follows that $\P(\calN(\calG,1) \cap \{K=x\}) \geq  \P(\calN(\calG,1) \cap \{K=y\})$ 
iff $ \n_x(x) - \n_x(y) \geq \n_x(y)$ iff $N=\n_x(x)  \geq 2\n_x(y)$ iff $\n_x(y) \leq N/2$.
This can be easily rephrased as  the statement of the proposition.
\end{proof}

Consider a tree $\calG$ with the vertex set $\calV$.
Let $\V$ be the family of all (unordered) pairs of subsets $\{\calV', \calV''\}$ 
of the vertex set $\calV$ such that $\calV'$ and $\calV''$ are vertex sets of  non-empty subtrees of $\calG$, $\calV' \cap \calV'' = \emptyset$ and $\calV' \cup \calV'' = \calV$. In other words, $\V$ is the family of all partitions of $\tree$ into two subtrees that can be generated by removing an edge.

\begin{lemma}\label{j24.1}
Consider a (deterministic) labeling $L$ of a tree $\tree$ that has exactly two peaks $y_1$ and $y_2$.

(i) There exist exactly two pairs 
$\{\calV', \calV''\}\in \V$ such that $y_1$ is the only peak  of $L$ restricted to $\calV'$ and $y_2$ is the only peak of $L$ restricted to $\calV''$. 

(ii) If the distance between $y_1$ and $y_2$ is greater than 2 then there exist at least three pairs $\{\calV', \calV''\}\in \V$ such that $y_1\in\calV'$ and $y_2\in\calV''$. Suppose that $\{\calV', \calV''\}\in \V$ is one of such pairs and it does not satisfy the condition stated in (i). Then
 either (a) $L$ restricted to $\calV'$ has exactly two peaks,  one of them located at $y_1$ and the other one adjacent to $\calV''$, and $L$ restricted to $\calV''$ has exactly one peak at $y_2$,
or (b) $L$ restricted to $\calV''$ has exactly two peaks,  one of them located at $y_2$ and the other one adjacent to $\calV'$, and $L$ restricted to $\calV'$ has exactly one peak at $y_1$. 
\end{lemma}

\begin{proof} (i)
Since $y_1$ and $y_2$ are peaks, the distance between them must be equal to or greater than 2.
 For $x,z\in\calV$, let $[x,z]$ denote the geodesic between $x$ and $z$, including both vertices. If there are  exactly two peaks $y_1$ and $y_2$ then there exists a unique $y_3\in [y_1,y_2]$ such that $y_1\ne y_3 \ne y_2$ and the labeling is monotone  on both $[y_1,y_3]$ and $[y_2,y_3]$.
Let $y_4 $ be the neighbor of $y_3$
in $[y_1,y_3]$ and let
$y_5 $ be the neighbor of $y_3$
in $[y_2,y_3]$. Note that $y_4$ could be $y_1$ or $y_5$ could be $y_2$.
Let $\{\calV', \calV''\}$ be subtrees obtained by removing the edge between $y_4$ and $y_3$ from $\tree$. 
Let $\{\wt\calV', \wt\calV''\}$ be subtrees obtained by removing the edge between $y_5$ and $y_3$ from $\tree$. 
It is easy to see that $\{\calV', \calV''\}$ and $\{\wt\calV', \wt\calV''\}$ are the only elements of $\V$ such that $y_1$ is the only peak in one of the subtrees in the pair and $y_2$ is the only peak in the other subtree in the same pair.

(ii) If the distance between $y_1$ and $y_2$ is greater then 2 then there are at least three edges along the geodesic  $[y_1,y_2]$. Removing an edge not adjacent to $y_3$  will generate a pair $\{\calV', \calV''\}\in \V$ with the properties specified in part (ii).
\end{proof}

The next result is related to the conjecture stated in Section \ref{intro}. It follows easily from the arguments used in the proof of \cite[Thm.~4.9]{BBPS} that if a random labeling of a large linear graph is conditioned to have exactly two peaks then the peaks are likely to be at a distance about one half of the length of the graph. Going in the opposite direction,
we will show that the ratio of the distance between the twin peaks and the tree diameter can be arbitrarily close to zero.

In the following proposition, if the random labeling has exactly two peaks, we will denote the location of the highest peak $K_1$ and the location of the other one $K_2$.

\begin{proposition}\label{y1.2}
Fix arbitrarily small $p>0$ and arbitrarily large $m$. There exists a tree with diameter greater than $m$ such that  $\P_2(K_1=x, K_2=y)$ attains the unique maximum for some $x$ and $y$ at the distance 3. Moreover, $\P_2( \dist(K_1,K_2) \geq 8 ) < p$.
\end{proposition}

\begin{proof}
Let $m > 100$, $n> (m!)^2$, and let the vertex set $\calV$ of the graph $\calG$ consist of points $x_1, x_2, \dots, x_n, y_1, y_2, \dots, y_n, z_1, z_2, \dots, z_m$. 
The only pairs of vertices connected by edges are $(z_k, z_{k+1})$ for $k\in[m-1]$, $(z_1, x_k)$ for $k\in[n]$ and $(z_4, y_k)$ for $k\in[n]$.  

First, we will find a lower estimate for  
\begin{align*}
\P_\calG&(\calN(\calG,2) \cap(  \{K_1=z_1, K_2=z_4\}
\cup \{K_2=z_1, K_1=z_4\})).
\end{align*}
Let $\calG'$ be the subgraph of $\calG$ consisting of vertices $z_1,z_2,x_1, x_2, \dots, x_n$. Let $\calG''$ be the subgraph of $\calG$ consisting of all the remaining vertices. Let $K'_1$ be the highest peak of a random labeling of $\calG'$ and let $K''_1$ be the highest peak of a random labeling of $\calG''$. We have
\begin{align*}
\P_\calG&(\calN(\calG,2) \cap(  \{K_1=z_1, K_2=z_4\}
\cup \{K_2=z_1, K_1=z_4\})) \\
&=
\P_\calG(\calN(\calG',1) \cap\{K'_1 = z_1\} \cap \calN(\calG'',1) \cap\{K''_1= z_4\})\\
&= \P_{\calG'}(\calN(\calG',1) \cap\{K'_1 = z_1\}) 
\P_{\calG''}(\calN(\calG'',1) \cap\{K''_1= z_4\}).
\end{align*}
We have equality on the second line above, despite Lemma \ref{j24.1} (ii), because if $z_1$ and $z_4$ are peaks of a labeling of $\calG$ then the same labeling restricted to $\calG'$ cannot have a peak at $z_2$ and the  labeling restricted to $\calG''$ cannot have a peak at $z_3$.

By \eqref{m21.1} applied  to $\calG'$,
\begin{align*}
\P_{\calG'}(\calN(\calG',1) \cap\{K'_1 = z_1\}) = \frac 1 {n+2}.
\end{align*}
The same formula applied  to $\calG''$  yields
\begin{align*}
\P_{\calG''}(\calN(\calG'',1) \cap\{K''_1= z_4\}) &= \frac 1 {(n+m-2)(m-4)(m-5) \dots 2 \cdot 1}\\
&= \frac 1 {(n+m-2)(m-4)!}.
\end{align*}
It follows that
\begin{align}\label{m21.2}
\P_\calG&(\calN(\calG,2) \cap(  \{K_1=z_1, K_2=z_4\}
\cup \{K_2=z_1, K_1=z_4\}))\\
&\geq \frac 1 {(n+2)(n+m-2)(m-4)!}.\nonumber
\end{align}

Next we will find an upper bound for $\P_\calG(\calN(\calG,2) \cap  \{K_1=v_1, K_2=v_2\})$,
assuming that $(v_1,v_2) \ne (z_1, z_4)$ and $(v_1,v_2) \ne (z_4, z_1)$. Let $\bfG$ be the family of all pairs $(\calG', \calG'')$ of disjoint  subtrees  of $\calG$, with vertex sets $\calV'$ and $\calV''$, such that $\calV'\cup \calV''=\calV$, $v_1\in \calV'$ and $v_2\in \calV''$.
Let $K'_1$ be the highest peak  of a random labeling of $\calG'$ and let $K'_2$ be the second highest peak  of  $\calG'$; similar notation will be used for $\calG''$. We will use Lemma \ref{j24.1}.

Consider a pair $(\calG', \calG'') \in \bfG$.
Note that if $z_1$ is in $\calG'$ then at least $n-1$ of vertices $x_1, x_2, \dots, x_n$ are also in $\calG'$. 
If $z_4$ is in $\calG'$ then at least $n-1$ of vertices $y_1, y_2, \dots, y_n$ are also in $\calG'$. 
Analogous claims hold for $z_1$ and $\calG''$, and for $z_4$ and $\calG''$.

Suppose that $z_1$ and $z_4$ are in two different graphs $\calG'$ and $\calG''$.
Suppose without loss of generality that  $z_1$ is in $\calG'$ and $z_4$ is in $\calG''$. 
Recall that $(v_1,v_2) \ne (z_1, z_4)$.
Consider the case when $v_1 \ne z_1$. 
Note that
 $\n^{\calG'}_{v_1}(v_1) \geq n$ and $ \n^{\calG'}_{v_1}(z_1)\geq n-1$ so, by \eqref{m21.1},
\begin{align*}
\P_{\calG'}(\calN(\calG',1) \cap \{K'_1 = v_1\}) \leq \frac 1 {n(n-1)}.
\end{align*}
Since
 $\n^{\calG''}_{v_2}(v_2) \geq n$ so, by \eqref{m21.1},
\begin{align*}
\P_{\calG''}(\calN(\calG'',1) \cap \{K''_1 = v_2\}) \leq \frac 1 {n}.
\end{align*}
A completely analogous argument shows that, if $v_1 = z_1$ and $v_2 \ne z_4$ then
\begin{align}
\P_{\calG'}(\calN(\calG',1) \cap \{K'_1 = v_1\}) &\leq \frac 1 {n},
\nonumber\\
\P_{\calG''}(\calN(\calG'',1) \cap \{K''_1 = v_2\}) &\leq \frac 1 {n(n-1)}.
\label{j25.2}
\end{align}
We conclude that if $z_1\in\calV'$, $z_4\in\calV''$ and
 $(v_1,v_2) \ne (z_1, z_4)$ then,
\begin{align}\label{j26.10}
\P_{\calG}(\calN(\calG',1) \cap \calN(\calG'',1) \cap
\{K'_1 = v_1,K''_1 = v_2\}) \leq  \frac 1 {n^2(n-1)}.
\end{align}

Recall the possibilities listed in Lemma \ref{j24.1}. 
We will estimate the probability of exactly two peaks in one of the subgraphs.

We are returning to the case when $z_1\in\calV'$, $z_4\in\calV''$,
 $(v_1,v_2) \ne (z_1, z_4)$ and $v_1\ne z_1$. 
Since $\calV'$ contains at least $n$ vertices,  $\P_{\calG'}(K'_1=v_1)\leq 1/n$. If the labeling has exactly two peaks then there are at least $n-2$ vertices 
$x'_1, \dots , x'_{n-2}$ among $x_1, \dots , x_n$, such that $x'_k\in \calV'$ and $x'_k$ is not a peak, for $k=1,\dots, n-2$. 
 Therefore, the label of $z_1$ must be larger than the labels of all $x'_k$'s. 
Note that $n \leq |\calV'| \leq n+3$.
Assuming that $K'_1 = v_1$, the probability that the label of $z_1$ is  
larger than the labels of all $x'_k$'s  is bounded by $5/(n-1)$. Hence,
\begin{align*}
\P_{\calG'}(\calN(\calG',2) \cap
\{K'_1 = v_1\}) \leq  \frac 5{n(n-1)}.
\end{align*}
The graph $\calG''$ has at least $n$ vertices so $\P_{\calG''}(K''_1=v_2)\leq 1/n$. We see that if $z_1\in\calV'$, $z_4\in\calV''$,
 $(v_1,v_2) \ne (z_1, z_4)$ and $v_1 \ne z_1$ then,
\begin{align}\label{j26.11}
\P_{\calG}(\calN(\calG',2) \cap \calN(\calG'',1) \cap
\{K'_1 = v_1,K''_1 = v_2\}) \leq  \frac 5 {n^2(n-1)}.
\end{align}

The next case to consider is when $z_1\in\calV'$, $z_4\in\calV''$,
 $(v_1,v_2) \ne (z_1, z_4)$ and $v_2\ne z_4$. 
Since $\calV''$ contains at least $n$ vertices,  $\P_{\calG''}(K''_1=v_2)\leq 1/n$. If the labeling has exactly two peaks then there are at least $n-2$ vertices 
$y'_1, \dots , y'_{n-2}$ among $y_1, \dots , y_n$, such that $y'_k\in \calV''$ and $y'_k$ is not a peak, for $k=1,\dots, n-2$. 
 Therefore, the label of $z_4$ must be larger than the labels of all $y'_k$'s. 
Note that $n \leq |\calV''| \leq n+m$.
Assuming that $K''_1 = v_2$, the probability that the label of $z_4$ is  
larger than the labels of all $y'_k$'s  is bounded by $(m+2)/(n-1)$. Hence,
\begin{align*}
\P_{\calG''}(\calN(\calG'',2) \cap
\{K''_1 = v_2\}) \leq  \frac {m+2}{n(n-1)}.
\end{align*}
The graph $\calG'$ has at least $n$ vertices so $\P_{\calG'}(K'_1=v_1)\leq 1/n$. We see that if $z_1\in\calV'$, $z_4\in\calV''$,
 $(v_1,v_2) \ne (z_1, z_4)$ and $v_2 \ne z_4$ then,
\begin{align}\label{j26.12}
\P_{\calG}(\calN(\calG',1) \cap \calN(\calG'',2) \cap
\{K'_1 = v_1,K''_1 = v_2\}) \leq  \frac {m+2} {n^2(n-1)}.
\end{align}

Next suppose that $z_1$ and $z_4$ are in the same of two graphs $\calG'$ and $\calG''$. Without loss of generality, suppose  that they are in $\calG'$.  
It is easy to see that
 $\n^{\calG'}_{v_1}(v_1) \geq n$ and $\n^{\calG'}_{v_1}(x) \geq n$ for at least two other $x$ in the set $\{z_1,z_2,z_3,z_4\}$ so, by \eqref{m21.1},
\begin{align*}
\P_{\calG'}(K'_1 = v_1) \leq \frac 1 {n^3}.
\end{align*}
The maximum of $1/n^3$ and the bounds in \eqref{j26.10}, \eqref{j26.11} and \eqref{j26.12} is less than $\frac{m+2 }{n^2(n-1)}$ for $m>3$. 
We obtain for any fixed $(\calG', \calG'') \in \bfG$,
\begin{align*}
\P_{\calG}(\calN(\calG,2)\cap\{K'_1 = v_1, K''_1 = v_2\}) \leq  \frac{m+2 }{n^2(n-1)}.
\end{align*}
 It is easy to see that for any $v_1$ and $v_2$ which are not adjacent,
the family $\bfG$ has at most $m$ elements. It follows that
\begin{align*}
\P_{\calG}(\calN(\calG,2)\cap\{K_1=v_1, K_2=v_2\}) 
&\leq \sum_{(\calG', \calG'') \in \bfG}
\P_{\calG}(\calN(\calG,2)\cap\{K'_1 = v_1, K''_1 = v_2\})\\
&\leq m \frac{m+2 }{n^2(n-1)} \leq  \frac {m(m+2)} {n(n-1) (m!)^2} .
\end{align*}
Comparing this bound to \eqref{m21.2}, we conclude that $\P_2(K_1=x, K_2=y)$ is maximized
at $(x,y)=(z_1,z_4)$ or $(x,y)=(z_4,z_1)$, for all large $m$.

\medskip

We will now strengthen our estimates assuming that $\dist(v_1,v_2) \geq 8$. 

First suppose that $z_1$ and $z_4$ are in the same of two graphs $\calG'$ and $\calG''$. Without loss of generality, suppose  that they are in $\calG'$.  
Then $\n^{\calG'}_{v_1}(x) \geq n-1$ for all $x$ in the set $\{z_1,z_2,z_3,z_4\}$ (this is true whether $v_1$ belongs to this set or not) so, by \eqref{m21.1},
\begin{align}\label{f18.1}
\P_{\calG'}(K'_1 = v_1) \leq \frac 1 {(n-1)^4}.
\end{align}
Since $z_1$ is in $\calG'$,
the condition $\dist(v_1,v_2) \geq 8$ implies that $v_2=z_k$ for some $k\geq 7$.
It follows that
there are at most $6nm^2$ families $(v_1,\calG',v_2, \calG'')$ such that
$\dist(v_1,v_2) \geq 8$ and $z_1$ and $z_4$ are in the same of the two graphs $\calG'$ and $\calG''$.

Suppose that $z_1$ and $z_4$ are in two different graphs $\calG'$ and $\calG''$. Consider the case when $z_1$ is in $\calG'$. 
Recall that $\dist(v_1,v_2) \geq 8$ implies that $v_2=z_k$ for some $k\geq 7$.
We have
 $\n^{\calG''}_{v_2}(v_2) \geq n$ and $ \n^{\calG''}_{v_2}(z_k) \geq n$ for $k=4,5,6$, so, by \eqref{m21.1},
\begin{align}\label{j26.16}
\P_{\calG''}(\calN(\calG'',1)\cap\{K''_1 = v_2\}) \leq \frac 1 {n^4}.
\end{align}
There are at most $18nm$ families $(v_1,\calG',v_2, \calG'')$ such that
$\dist(v_1,v_2) \geq 8$ and $z_1$ and $z_4$ are in two different graphs $\calG'$ and $\calG''$.

We now analyze the case when there are exactly two peaks in $\calG''$.
Since $\calV''$ contains at least $n$ vertices,  $\P_{\calG''}(K''_1=v_2)\leq 1/n$. If the labeling has exactly two peaks then there are at least $n-2$ vertices 
$y'_1, \dots , y'_{n-2}$ among $y_1, \dots , y_n$, such that $y'_k\in \calV''$ and $y'_k$ is not a peak, for $k=1,\dots, n-2$. 
 Therefore, the label of $z_4$ must be larger than the labels of all $y'_k$'s. 
Recall that $n \leq |\calV''| \leq n+m$.
Assuming that $K''_1 = v_2$, the probability that the label of $z_4$ is  
larger than the labels of all $y'_k$'s  is bounded by $(m+2)/(n-1)$. Hence,
\begin{align}\label{j26.15}
\P_{\calG''}(\calN(\calG'',2) \cap
\{K''_1 = v_2\}) \leq  \frac {m+2}{n(n-1)}.
\end{align}
The graph $\calG'$ has at least $n$ vertices so $\P_{\calG'}(K'_1=v_1)\leq 1/n$. If $v_1 = z_1$ then, using \eqref{j26.15},
\begin{align}\label{j26.17}
\P_{\calG}(\calN(\calG',1) \cap \calN(\calG'',2) \cap
\{K'_1 = v_1 = z_1, K''_1 = v_2\}) \leq  \frac {m+2}{n^2(n-1)}.
\end{align}
There are at most $6m$ families $(v_1,\calG',v_2, \calG'')$ such that
$\dist(v_1,v_2) \geq 8$, $z_1$ and $z_4$ are in two different graphs $\calG'$ and $\calG''$,  and $\calG''$ and $v_1=z_1$ or $v_2 = z_1$.

If $v_1 \ne z_1$ then
 $\n^{\calG'}_{v_1}(v_1) \geq n$ and $ \n^{\calG'}_{v_1}(z_1) \geq n$,  so, by \eqref{m21.1},
\begin{align*}
\P_{\calG'}(\calN(\calG',1)\cap\{K'_1 = v_1\}) \leq \frac 1 {n^2},
\end{align*}
and, in view of \eqref{j26.15},
\begin{align}\label{j26.18}
\P_{\calG}(\calN(\calG',1) \cap \calN(\calG'',2) \cap
\{K'_1 = v_1 , K''_1 = v_2\}) \leq  \frac {m+2}{n^3(n-1)}.
\end{align}
There are at most $18nm$ families $(v_1,\calG',v_2, \calG'')$ such that
$\dist(v_1,v_2) \geq 8$ and $z_1$ and $z_4$ are in two different graphs $\calG'$ and $\calG''$.

We now appeal to Lemma \ref{j24.1} and combine \eqref{f18.1}, \eqref{j26.16}, \eqref{j26.17} and \eqref{j26.18} to conclude that, for large $m$,
\begin{align*}
&\sum_{\dist(v_1,v_2) \geq 8}
\P_{\calG}(\calN(\calG,2) \cap
\{K_1 = v_1 , K_2 = v_2\}) \\
&\leq  
6 n m^2 \frac 1 {(n-1)^4}
+18nm\frac 1 {n^4}
+ 6 m\frac {m+2}{n^2(n-1)} 
+ 18nm 
\frac {m+2}{n^3(n-1)}
\leq \frac{50 m (m+2)}{ (n-1)^3}.
\end{align*}
 Therefore, in view of \eqref{m21.2},
\begin{align*}
\P_2( \dist(K_1,K_2) \geq 8 ) & \leq 
\frac{\sum_{\dist(v_1,v_2) \geq 8}
\P_{\calG}(\calN(\calG,2) \cap
\{K_1 = v_1 , K_2 = v_2\})}
{\P_\calG(\calN(\calG,2) \cap(  \{K_1=z_1, K_2=z_4\}
\cup \{K_2=z_1, K_1=z_4\}))} \\
& \leq \frac{50 m (m+2)}{ (n-1)^3} (n+1)(n+m-2)(m-4)!.
\end{align*}
The last expression goes to 0 when $n\to \infty$ so for any $p>0$, we have $\P_2( \dist(K_1,K_2) \geq 8 ) < p$,
if $n$ is sufficiently large.
\end{proof}

\section{Twin peaks on regular trees}\label{regtrees}

This section investigates twin peaks on $(d+1)$-regular trees of constant depth. 
 
\begin{definition} 
A finite rooted tree will be called a $(d+1)$-regular tree of depth $k$ if the root, denoted $v_*$, has $d+1$  children, every child has further $d$  children and so on, continuing till $k$ generations. That is, the distance of each leaf to the root is exactly $k$, and the degree of every non-leaf vertex is $d+1$. 
\end{definition}

Our main result in this section is not as complete as we wish it had been.
Theorem \ref{y23.1} (i) requires  the assumption that $d\geq 10$; we believe that the assumption could be weakened.  

 We will briefly outline heuristic reasons why finding the location of the two peaks on a $(d+1)$-regular tree seems to be particularly challenging. If there are only two peaks on a tree graph, the graph can be divided into two subtrees such that the labeling restricted to each of the subtrees has only one or two peaks (see Lemma \ref{j24.1}). If a $(d+1)$-regular tree, for $d \geq 2$, of constant depth is divided into two subtrees then it is easy to see that the centroid of the smaller subtree is adjacent to a vertex in the other subtree. Hence, in  view of Proposition \ref{ap1.1}, one would expect the top peak in  the smaller of the two subtrees to be close to the other subtree. This seems to contradict the tendency for  two peaks to be far away on ``well-structured'' graphs (see the remarks preceding  Proposition \ref{y1.2}).

\begin{theorem}\label{y23.1} Consider a random labeling of a $(d+1)$-regular tree $\tree$ of depth $k$. Let the locations of the highest and second highest peaks of the labeling be denoted $K_1$ and $K_2$, respectively. 

(a) If  $d\geq 10 $ and $k\geq2$ then
the function $(y_1,y_2) \to \P_2(K_1= y_1, K_2= y_2)$ takes the maximum value only if  $y_1=v_*$. 

(b) If $d\geq 3$, $k\geq 2$ and $m\geq  \sqrt{8k} $
 then
\begin{align*}
\P_2\left(\min_{i=1,2} \dist(K_i, v_*) \geq m\right) 
\leq 8 \exp(-  (m/2) \log (d-1)).
\end{align*}

(c) For every $d\geq 3$ there exists $c_1>0$ such that for all $k\geq 2$ and $y_1,y_2\in \calV$ with $\dist(y_i,v_*) \geq 2$ for $i=1,2$, we have
\begin{align*}
\P_2(K_1= v_*, K_2= y_1) \geq c_1 \P_2(K_1= v_*, K_2= y_2).
\end{align*}

(d) For every $d\geq 3$ there exist $c$ and $c'$ such that for $k\geq 2$ and $n\geq 2\sqrt{8k} $,
\begin{align*}
\P_2&\left(|\dist(K_1, K_2) - k | \geq n\right) \\
& \leq  8 \exp(- (n/4) \log (d-1))
+ c \exp(-(n/2)\log d)\\
& \leq  c' \exp(- (\sqrt{k}/64) \log d).
\end{align*}

\end{theorem}

\begin{proof}(a)
{\it Step 1}.
First we will prove a general estimate
similar to the ``total probability formula.''
Suppose that $A_1, A_2, \dots, A_n$ satisfy the following condition,
\begin{align*}
A_i \cap A_j \cap A_m &= \emptyset \quad \text{  if }  i\ne j \ne m \ne i.
\end{align*}
 These events need not be pairwise disjoint but no triplet has a nonempty intersection.
Suppose that for some $c_1 >0$, some events $B$ and $C$, and all $i$, 
\begin{align*}
\P(B\mid A_i)\leq c_1 \P(C\mid A_i).
\end{align*}
Assume that $B \subset \bigcup_{i=1}^n A_i$.
We will show that 
\begin{align}\label{y23.2}
\P(B) \leq 2 c_1 \P(C).
\end{align}
The proof is contained in the following calculation,
\begin{align*}
\P(B) 
& = \P\left(B \cap \bigcup_{i=1}^n A_i \right)
 = \P\left( \bigcup_{i=1}^n (B \cap A_i) \right)
\leq
 \sum_{i=1}^n \P(B\cap A_i) 
=
 \sum_{i=1}^n \P(B\mid A_i) \P(A_i) \\
&\leq \sum_{i=1}^n c_1\P(C\mid A_i) \P(A_i) 
=\sum_{i=1}^n c_1\P(C\cap A_i)  
=c_1
 \sum_{i=1}^n \E \bone_{C \cap A_i} \\
&=c_1 \E\left (\sum_{i=1}^n \bone_{C \cap A_i} \right) 
\leq c_1\E\left ( 2 \cdot \bone_{C } \right) 
= 2c_1 \P(C).
\end{align*}

\medskip
{\it Step 2}.
If we remove an edge (but retain all vertices) of $\tree$, we obtain two subtrees. Consider a subtree $\tree'$ constructed in this way and let $v_1$ be the vertex in $\tree'$ closest to  the root $v_*$ in $\tree$. If $v_*$ is in $\tree'$ then $v_ 1=v_*$. Consider a random labeling of $\tree'$ conditioned on having exactly one peak. Let $K$ denote the position of the peak. 

We create a branching structure in $\tree'$ by declaring that $v_1$ is the ancestor of all vertices in $\tree'$.
Suppose that a vertex $x_1\in \tree'$ is a parent of $x_2\in \tree'$. 
Recall the notation $\n^{\tree'}_x(y)$ from 
Definition \ref{j24.2}.
It is easy to see that $\n^{\tree'}_{x_2}(x_1) \geq (d-1) \n^{\tree'}_{x_1}(x_2) $. It follows from \eqref{y24.1} that
\begin{align}\label{y27.5}
\P_1(K= x_1) \geq (d-1) \P_1(K=x_2).
\end{align}
By induction, for any $x\in \tree'$ such that $x\ne v_1$,
\begin{align}\label{y24.2}
\P_1(K= v_1) \geq (d-1) \P_1(K=x).
\end{align}

We will also need sharper versions of the above estimates.
Suppose that $v_1 = v_*$, i.e., $v_*\in \tree'$.
Suppose that a vertex $x_1\in \tree'$ is a parent of $x_2\in \tree'$
and the distance from $x_1$ to $v_*$ is $j$. 
Then $\n^{\tree'}_{x_2}(x_1) \geq (d-1)^{j+1} \n^{\tree'}_{x_1}(x_2) $. It follows from \eqref{y24.1} that
\begin{align}\label{y29.5}
\P_1(K= x_1) \geq (d-1)^{j+1} \P_1(K=x_2).
\end{align}
By induction, for any $x\in \tree'$ such that $\dist(x, v_*) = m\geq 1$,
\begin{align}\label{y29.6}
\P_1(K= v_*) \geq \P_1(K=x) \prod_{j=0}^{m-1} (d-1)^{j+1} 
= (d-1)^{m(m+1)/2} \P_1(K=x).
\end{align}

\medskip
{\it Step 3}.
Let $\V$ be the family of all (unordered) pairs of subsets $\{\calV', \calV''\}$ 
of the vertex set $\calV$ such that $\calV'$ and $\calV''$  are the vertex sets of non-empty subtrees, $\calV' \cap \calV'' = \emptyset$ and $\calV' \cup \calV'' = \calV$. In other words, $\V$ is the family of all partitions of $\tree$ into two subtrees that can be generated by removing an edge.
Graphs corresponding to $\calV'$ and $\calV''$ will be denoted $\calG'$ and $\calG''$.

Consider the following conditions on vertices $y_1,y_2\in \calV$,

(A1) The distance between $y_1$ and $y_2$ is  larger than 2.

(A2) The distance between $y_1$ and $y_2$  is equal to 2 and the vertex  between $y_1$ and $y_2$ is not $v_*$. We will denote this vertex by $y_3$.

In case (A1), let $\V(y_1,y_2)\subset \V$ be the family of all $\{\calV', \calV''\}$ such that $y_k\in \calV_k$ and $y_k$ is not adjacent to a vertex in $\calV_{3-k}$, for $k=1,2$. 

 In  case (A2), we define the family $\V(y_1,y_2)$ as a set containing only one pair $\{\calV', \calV''\}$  constructed as follows. If we remove $y_3$ then the  graph is split into two subgraphs $\calV_1$ and $\calV_2$, labeled so that $y_k\in \calV_k$ for $k=1,2$. Suppose that $v_*\in \calV_1$. Then we let $\calV'' = \calV_2 \cup \{y_3\}$ and $\calV' = \calV_1$.
If $v_*\in \calV_2$ then we let $\calV' = \calV_1 \cup \{y_3\}$ and $\calV'' = \calV_2$.

For $\{\calV', \calV''\}\in \V$ and a random labeling,
let $A(\calV', \calV'') = \calN(\calG',1) \cap \calN(\calG'',1)$, i.e.,  $A(\calV', \calV'')$ is the event that the labeling restricted to $\calV_i$ has exactly one peak, for both $i=1$ and $i=2$.

It follows from Lemma \ref{j24.1} that in both cases, (A1) and  (A2), 
\begin{align}\label{u20.1}
\calN(\calG,2) \cap \{K_1=y_1, K_2 = y_2\} \subset \bigcup_{\{\calV', \calV''\} \in \V(y_1,y_2)  }
A(\calV', \calV'').
\end{align}

\medskip
{\it Step 4}.
Let
\begin{align*}
B_1 &= \calN(\calG,2) \cap \{K_1=y_1, K_2 = y_2\} ,\qquad
B_2 = \calN(\calG,2) \cap \{K_1=y_2, K_2 = y_1\} .
\end{align*}

Assume that (A1) or (A2) holds.
Suppose that 
$\{\calV', \calV''\}\in \V(y_1,y_2)$. 
Recall that $y_1 \in \calV'$.
We will show  that if $v_* \in \calV'$  then
\begin{align}\label{y26.1}
\P(B_1\cap A(\calV', \calV'')) \geq \P(B_2\cap A(\calV', \calV'')).
\end{align}
Let $F$ be the event that the highest label is given to a vertex in $\calV'$.
Let $K'_1$ be the location of the highest peak in $\calG'$ and let $K''_1$ be the location of the highest peak in $\calG''$.

Since $v_* \in \calV'$, we have $|\calV'| > |\calV''|$, and thus, $p_1 :=\P(F) > 1/2$. Note that
\begin{align*}
\P&(\calN(\calG',1) \cap \calN(\calG'',1) \cap \{K'_1 = y_1, K''_1 =y_2\}
\mid F)\\
&=
\P(\calN(\calG',1) \cap \calN(\calG'',1) \cap \{K'_1 = y_1, K''_1 =y_2\}
\mid F^c),
\end{align*}
and call the common value of the two conditional probabilities $p_2$.

Since $y_2$ is not adjacent to a vertex in $\calV'$, the event $\{B_1\cap A(\calV', \calV'')\}$ is the same as the intersection of the events (i) the highest label is given to a vertex in $\calV'$, (ii) there is a single peak at $y_1$ in $\calV'$ and (iii) there is a single peak at $y_2$ in $\calV''$.
Hence $\P(B_1\cap A(\calV', \calV'')) = p_1 p_2$.

If the event $\{B_2\cap A(\calV', \calV'')\}$ holds then the  intersection of the following events also holds: (i)  the highest label is given to a vertex in $\calV''$, (ii) there is a single peak at $y_1$ in $\calV'$ and (iii) there is a single peak at $y_2$ in $\calV''$ (but $\{B_2\cap A(\calV', \calV'')\}$ is not equal to the intersection of (i)-(iii)). This implies that
$\P(B_2\cap A(\calV', \calV'')) \leq (1-p_1) p_2$. Since $p_1> 1/2$, 
\begin{align*}
\P(B_1\cap A(\calV', \calV''))= p_1 p_2 > (1-p_1) p_2 \geq \P(B_2\cap A(\calV', \calV'')),
\end{align*}
so \eqref{y26.1} is proved. 

\medskip
{\it Step 5}.
Consider  $y_1,y_2\in \calV$ such that either (A1) or (A2) is satisfied  and none of these vertices is the root.
Let
\begin{alignat}{2}
&B_1 = \calN(\calG,2) \cap \{K_1=y_1, K_2 = y_2\} ,&&\qquad
C_1 = \calN(\calG,2) \cap \{K_1=v_*, K_2 = y_2\} ,\nonumber\\
&B_2 = \calN(\calG,2) \cap \{K_1=y_2, K_2 = y_1\} ,&&\qquad
C_2 = \calN(\calG,2) \cap  \{K_1=v_*, K_2 = y_1\} ,\nonumber\\
&B = B_1 \cup B_2,&&\qquad C= C_1 \cup C_2.
\label{y24.5} 
\end{alignat}
Consider the event $A(\calV', \calV'')$ for some $\{\calV', \calV''\}\in \V(y_1,y_2)$ and suppose that $v_* \in \calV'$.
If we condition on $ A(\calV', \calV'')$,
the distribution of the order statistics of labels in $\calV'$ is independent of the values of the labels in $\calV''$. Hence, \eqref{y24.2} shows that conditional on $ A(\calV', \calV'')$, the probability that the only peak in $\calV'$ is at $v_*$ is at least $d-1$ times larger than the probability that it is at $y_1$. 
Since $y_2$ is not adjacent to a vertex in $\calV'$,
if the only peak in $\calV''$ is at $y_2$, the only peak in $\calV'$ is at $v_*$ and the highest label is in $\calV'$ then $C_1 $ holds. This and \eqref{y26.1} imply that,
\begin{align}\label{y24.4}
\P(B\mid A(\calV', \calV'')) \leq 
2 \P(B_1\mid A(\calV', \calV'')) \leq \frac 2 {d-1}  \P(C_1\mid A(\calV', \calV'')).
\end{align}
By symmetry, if  $v_* \in \calV''$ then
\begin{align}\label{y26.2}
\P(B\mid A(\calV', \calV'')) \leq 
2 \P(B_2\mid A(\calV', \calV'')) \leq \frac 2 {d-1} \P(C_2\mid A(\calV', \calV'')).
\end{align}
It follows from  \eqref{y24.4}-\eqref{y26.2} that
for all $(\calV', \calV'')\in \V(y_1,y_2)$,
\begin{align}\label{y26.3}
\P(B\mid A(\calV', \calV''))   \leq \frac 2 {d-1} \P(C\mid A(\calV', \calV'')).
\end{align}
We will apply Step 1 with the family $\{A(\calV', \calV'')\}_{(\calV', \calV'')\in \V(y_1,y_2)}$ (for fixed $y_1$ and $y_2$) playing the role of the family $\{A_k\}$.
In view of \eqref{u20.1}, we see that \eqref{y23.2} and \eqref{y26.3} imply that
\begin{align*}
\P(B) \leq \frac 4 {d-1} \P(C).
\end{align*}
Recall events  $C_1$ and $C_2$ from the definition \eqref{y24.5} of $C$.
We have assumed that $d\geq 10$.
The last estimate implies that for at least one $i$ we must have
\begin{align}\label{y27.12}
\P(C_i) \geq \frac {d-1} 8 \P(B) > \P(B) \geq \P(\calN(\calG,2) \cap
\{K_1=y_1, K_2 = y_2\}).
\end{align}

\medskip

In preparation for  proofs of other  parts of the theorem, we present a stronger version of the last estimate under stronger assumptions. Suppose that for some $m\geq 2$ we have $\min_{i=1,2} \dist(y_i, v_*) \geq m $. 
Assume that $\{\calV', \calV''\}\in \V(y_1,y_2)$ and $v_* \in \calV'$.
If we condition on $ A(\calV', \calV'')$,
the distribution of the order statistics of labels in $\calV'$
(i.e., the joint distribution of the location of the highest label in $\calV'$, the location of the second highest label in $\calV'$, etc.) is independent of the values of the labels in $\calV'$.
 Hence, \eqref{y29.6} shows that conditional on $ A(\calV', \calV'')$, the probability that the only peak in $\calV'$ is at $v_*$ is at least $(d-1)^{-m(m+1)/2}$ times larger than the probability that it is at $y_1$. 
If the only peak in $\calV''$ is at $y_2$, the only peak in $\calV'$ is at $v_*$ and the highest label is in $\calV'$ then $K_1 = v_*$. We obtain using \eqref{y26.1},
\begin{align}\label{y29.8}
\P(B\mid A(\calV', \calV'')) \leq 
2 \P(B_1\mid A(\calV', \calV'')) \leq  2 (d-1)^{-m(m+1)/2} \P(C_1\mid A(\calV', \calV'')).
\end{align}
By symmetry, if $v_* \in \calV''$ then
\begin{align}\label{y29.9}
\P(B\mid A(\calV', \calV'')) \leq 
2 \P(B_2\mid A(\calV', \calV'')) \leq 2 (d-1)^{-m(m+1)/2} \P(C_2\mid A(\calV', \calV'')).
\end{align}
It follows from  \eqref{y29.8}-\eqref{y29.9} that
for all $\{\calV', \calV''\}\in \V(y_1,y_2)$,
\begin{align}\label{y29.10}
\P(B\mid A(\calV', \calV''))   \leq 2 (d-1)^{-m(m+1)/2} \P(C\mid A(\calV', \calV'')).
\end{align}
In view of \eqref{y23.2}, \eqref{u20.1} and \eqref{y29.10},
\begin{align}\label{y30.1}
\P(B) \leq 4 (d-1)^{-m(m+1)/2} \P(C)\leq 4 (d-1)^{-m(m+1)/2}
\P(\calN(\calG,2)).
\end{align}

\medskip
{\it Step 6}.
We will show that
 for any $y$ whose distance from the root is greater than or equal to 2,
\begin{align}\label{y27.11}
\P(\calN(\calG,2) \cap \{K_1= v_*, K_2= y\}) 
> \P(\calN(\calG,2) \cap \{K_1= y, K_2= v_*\}).
\end{align}

Let $\calV_i$, $1\leq i \leq d+1$, be the subtrees obtained by removing the root from $\tree$. Without loss of generality, we assume that $y\in\calV_1$.

Let $N = |\calV|$.
A random labeling of $\tree$ can be generated in the following way. First, we generate independent random labelings $L_j$ of $\calV_j$, for $1\leq j \leq d+1$. Then we assign a random label from the set $[N]$ to $v_*$ and divide randomly the labels remaining  in the set $[N]$ into $d+1$ (unordered) families $\Lambda_j$ of equal sizes,
independently from $L_j$'s. Next, labels in $\Lambda_j$ are placed at vertices in $\calV_j$ in such a way that the order structure is the same as for $L_j$, for every $j$. Let $L$ be the name of the resulting labeling of $\calG$.

Let $F_1$ denote the event that

(i) For $j\ne 1$, the labeling $L_j$ has only one peak and it is located at a vertex adjacent to $v_*$, and

(ii) $L_1$ has either only one peak at $y$ or it has exactly two peaks, one at $y$ and one at the neighbor of $v_*$.

It follows from Lemma \ref{j24.1} that
\begin{align}\label{j30.2}
\calN(\calG,2) \cap \{K_1= v_*, K_2= y\} &\subset F_1
\quad \text {  and}\quad
\calN(\calG,2) \cap \{K_1= y, K_2= v_*\}\subset F_1.
\end{align}

If $L(v_*) = N$ then (i)-(ii) are not only necessary but also sufficient conditions for the event that there are exactly two peaks at $v_*$ and $y$.
More formally,
\begin{align}\label{j30.1}
\{L(v_*) = N\} \cap F_1=
\{L(v_*) = N\} \cap
\calN(\calG,2) \cap \{K_1= v_*, K_2= y\} .
\end{align}

If $F_1$ holds, $L(v_*) = N-j$  and $L$ has exactly two peaks, with one of them at $v_*$ and the other in $\calV_1$, 
then all labels $N, N-1, \dots, N-j+1$ must be in $\calV_1$. Otherwise, because of (i), at least one of the numbers $N, N-1, \dots, N-j+1$ would be the label of a vertex adjacent to $v_*$ and, therefore, $v_*$ would not be a peak. Given $\{L(v_*) = N-j\} \cap F_1$, the probability that all labels $N, N-1, \dots, N-j+1$ are in $\calV_1$ is less than $1/(d+1)^j$ for $j>1$ and equal to $1/(d+1)$ for $j=1$.
We use \eqref{j30.2} and \eqref{j30.1} in the following calculation, 
\begin{align*}
\P&(\calN(\calG,2) \cap \{K_1= y, K_2= v_*\})
=\P(\calN(\calG,2) \cap \{K_1= y, K_2= v_*\} \cap F_1)\\
&= \sum_{j=1}^{N-1}  \P(\calN(\calG,2) \cap \{K_1= y, K_2= v_*,L(v_*) = N-j \} \cap F_1)\\
&= \sum_{j=1}^{N-1}  \P\Big(\calN(\calG,2) \cap F_1 \\
&\qquad\cap \{K_1= y, K_2= v_*,L(v_*) = N-j, \{N, N-1, \dots, N-j+1\}\subset L(\calV_1) \} \Big)\\
&\leq \sum_{j=1}^{N-1}  \P(
 \{L(v_*) = N-j, \{N, N-1, \dots, N-j+1\}\subset L(\calV_1) \}  \cap F_1)\\
&\leq \sum_{j=1}^{N-1} \frac 1 {(d+1)^j} \P( \{L(v_*) = N-j \}  \cap F_1)\\
&= \sum_{j=1}^{N-1} \frac 1 {(d+1)^j} \P( \{L(v_*) = N \}  \cap F_1)
< \frac 1 d \P( \{L(v_*) = N \}  \cap F_1)\\
& = \frac 1 d 
\P(\{L(v_*) = N\} \cap
\calN(\calG,2) \cap \{K_1= v_*, K_2= y\})\\
&= \frac 1 d 
\P(\calN(\calG,2) \cap \{K_1= v_*, K_2= y\}).
\end{align*}
Since $1/d < 1$ for $d\geq 2$, we conclude that \eqref{y27.11} holds.

\medskip
{\it Step 7}.
It follows from \eqref{y27.12} and \eqref{y27.11} that $\P(\calN(\calG,2) \cap \{K_1= y_1, K_2= y_2\})$ is maximized either when $y_1 = v_*$ or if $y_1$ and $y_2$ are neighbors of $v_*$.

Assume that $y_1$ and $y_2$ are neighbors of $v_*$ and $y_3$ is a neighbor of $y_2$, but $y_3 \ne v_*$. We will show that 
\begin{align}\label{y29.1}
\P(\calN(\calG,2) \cap \{K_1= v_*, K_2= y_3\}) 
> \P(\calN(\calG,2) \cap \{K_1= y_1, K_2= y_2\}).
\end{align}

The proof will use the same ideas as  the proof of Proposition \ref{ap1.1}. Recall the notation 
from Definition \ref{j24.2}.

The probability $\P(\calN(\calG,2) \cap \{K_1= v_*, K_2= y_3\})$ is the product of the following five factors, (i)-(v).

(i) The probability that the label of $v_*$ is the highest among all labels. This probability  is equal to $1/|\calV|$, where $\calV$ is the vertex set of $\tree$.

(ii) If we remove $v_*$ from $\tree$, we obtain $d+1$ subtrees. Let $\tree'$ be the subtree which contains $y_3$. Its vertex set will be denoted $\calV'$. The second factor is the probability that $y_3$ has the largest label in $\tree'$. This probability is equal to $1/|\calV'|$.

(iii) The third factor is the probability that the label of $y_2$ is larger than the labels of all of its descendants in $\tree'$, if we consider  $y_3$ to be the root of $\tree'$. This  probability is equal to $1/\n_{y_3}^{\tree'}(y_2)$.

(iv) The fourth factor is the probability that the label of $y_1$ is larger than the labels of all of its descendants in $\tree$, with the usual root  $v_*$. This  probability is equal to $1/\n^\calG_{v_*}(y_1)$.

(v) Just like in the proof of Proposition \ref{ap1.1}, we have to multiply by probabilities corresponding to all other vertices in $\tree$,
i.e., the last factor is $\prod_{x\ne v_*,y_1,y_2,y_3} 1/\n^\calG_{v_*}(x)$.

\smallskip
The probability $\P(\calN(\calG,2) \cap \{K_1= y_1, K_2= y_2\})$ is the product of the following five factors, (1)-(5).

(1) The probability that the label of $y_1$ is the highest among all labels. This probability  is equal to $1/|\calV|$, where $\calV$ is the vertex set of $\tree$.

(2) If we remove $y_1$ from $\tree$, we obtain $d+1$ subtrees. Let $\tree''$ be the subtree which contains $y_2$. Its vertex set will be denoted $\calV''$. The second factor is the probability that $y_2$ has the largest label in $\tree''$. This probability is equal to $1/|\calV''|$.

(3) The third factor is the probability that the label of $v_*$ is larger than the labels of all of its descendants in $\tree''$, if we consider  $y_2$ to be the root of $\tree''$. This  probability is equal to $1/\n_{y_2}^{\tree''}(v_*)$.

(4) The fourth factor is the probability that the label of $y_3$ is larger than the labels of all of its descendants in $\tree$, with the usual root  $v_*$. This  probability is equal to $1/\n^\calG_{v_*}(y_3)$.

(5) We have to multiply by probabilities corresponding to all other vertices in $\tree$,
i.e., the last factor is $\prod_{x\ne v_*,y_1,y_2,y_3} 1/\n^\calG_{v_*}(x)$.

Note that the factors described in (i) and (1) are identical. The same applies to (v) and (5). Hence,  it will suffice to show that
\begin{align}\label{y28.1}
\frac 1 {|\calV'|}\ 
\frac 1 {\n_{y_3}^{\tree'}(y_2)} \ 
\frac 1 {\n^\calG_{v_*}(y_1)}
>
\frac 1 {|\calV''|} \ \frac 1 {\n_{y_2}^{\tree''}(v_*)} \ \frac 1 {\n^\calG_{v_*}(y_3)}.
\end{align}
The above inequality is equivalent to each of the following inequalities.
\begin{align*}
&|\calV'| \cdot \n_{y_3}^{\tree'}(y_2) \cdot \n^\calG_{v_*}(y_1)
< |\calV''| \cdot \n_{y_2}^{\tree''}(v_*) \cdot \n^\calG_{v_*}(y_3),\\
&\frac { d^k -1} {d-1}
\left( \frac { d^{k-1} -1} {d-1}\cdot (d-1) + 1 \right)
\frac { d^k -1} {d-1} \\
&\qquad\qquad<
\left( \frac { d^{k} -1} {d-1} \cdot d + 1 \right)
\left( \frac { d^{k} -1} {d-1} \cdot (d-1) + 1 \right)
\frac { d^{k-1} -1} {d-1},\\
&\frac{(d^k -1)  d^{k-1} (d^k -1) }
{(d-1)^2} <
\frac{(( d^{k} -1)d +d -1)  d^{k} ( d^{k-1} -1)}
{(d-1)^2},\\
&(d^k -1)   (d^k -1) <
( d^{k+1}  -1) d ( d^{k-1} -1),
\end{align*}
\begin{align}\label{y29.2}
- (d-1) (d^k(d^k - d -1) +1) <0.
\end{align}
It is easy to see that
\eqref{y29.2} is true for all $d\geq 3$ and $k\geq 2$. It follows that \eqref{y28.1} is also true. This completes the proof of \eqref{y29.1} and, therefore, the proof of part (a) of the theorem.

\medskip
(b)
We have
\begin{align}\label{j23.1}
\P&\left(\calN(\calG,2) \cap\left\{\min_{i=1,2} \dist(K_i, v_*) \geq m \right\}\right) \\
&= \frac 12 \sum_{y_1,y_2 \atop \min_{i=1,2} \dist(y_i, v_*) \geq m}
\P(\calN(\calG,2) \cap ( \{K_1=y_1, K_2 = y_2\} \cup
 \{K_1=y_2, K_2 = y_1\})).\nonumber
\end{align}
The factor $1/2$ on the second line appears because we include both $(y_1,y_2)$ and $(y_2,y_1)$ in the sum.

The number of vertices in $\tree$ is $(d+1) (d^{k}-1)/(d-1)+1$ so, assuming that $d\geq 3$, the number of pairs $(y_1,y_2)$ is bounded by 
\begin{align}\label{j23.2}
\left(\frac{(d+1) (d^{k}-1)}{d-1}+1\right)\left(\frac{(d+1) (d^{k}-1)}{d-1}\right) \leq 4d^{2k}.
\end{align}

It follows from \eqref{y30.1} that for $m\geq 2$,
\begin{align*}
&\P(\calN(\calG,2) \cap ( \{K_1=y_1, K_2 = y_2\} \cup
 \{K_1=y_2, K_2 = y_1\}))\\
&\leq  4 (d-1)^{-m(m+1)/2}\P(\calN(\calG,2)).
\end{align*}
This, \eqref{j23.1} and \eqref{j23.2} imply that
\begin{align}\label{j23.3}
\P\left(\calN(\calG,2) \cap\left\{\min_{i=1,2} \dist(K_i, v_*) \geq m \right\}\right) \leq 2 d^{2k} 4 (d-1)^{-m(m+1)/2}\P(\calN(\calG,2)) .
\end{align}

If $d\ge3 $ and $k \leq\frac 18 m^2 $ then $k \leq\frac 14 m^2 \log (d-1)/\log d$ and
\begin{align*}
2 d^{2k} 4 (d-1)^{-m(m+1)/2} 
&= 8 \exp ( \log (d^{2k}  (d-1)^{-m(m+1)/2}))\\
&= 8 \exp ( 2k \log d  -\frac 12 m(m+1) \log (d-1))\\
&\leq 8 \exp (   -(m/2) \log (d-1)).
\end{align*}
We use this and \eqref{j23.3} to see that
 if $d\geq 3$, $k\geq 2$ and
$ m \geq \sqrt{8k} $
 then
\begin{align*}
\P\left(\calN(\calG,2) \cap\left\{\min_{i=1,2} \dist(K_i, v_*) \geq m \right\}\right) &\leq 
8 \exp (   -(m/2) \log (d-1))\P(\calN(\calG,2)) .
\end{align*}
This proves  part (b) of the theorem.

\medskip
(c) The argument will be based on the comparison of $\P(\calN(\calG,2) \cap \{K_1=v_*, K_2 = y\})$ and $\P(\calN(\calG,1) \cap\{ K_1=v_*\})$. 

 Suppose that $y\in\calV$ and  $m:=\dist(y, v_*) \geq 2$.
Let $\Gamma$ be  the geodesic between $y$ and $v_*$ and let
 $z_j \in \Gamma$, $0\leq j \leq m$,  be such that $\dist(z_j, y) = j$. Note that it is possible that $z_2=v_*$.
Suppose that $\calV'$ and $\calV''$ are obtained by removing the edge between $z_1$ and $z_2$ from $\tree$; we assume without loss of generality  that
 $v_*\in\calV'$. The corresponding graphs will be denoted $\calG'$ and $\calG''$.
If  a labeling $L$ of $\tree$ has the following properties: (i) $L$ restricted to $\calV'$ has only one peak at $v_*$,  (ii) $L$ restricted to $\calV''$ has only one peak at $y$, and  (iii) the largest label is in $\calV'$, then $L$ has exactly two peaks at $v_*$ and $y$. Since $v_* \in \calV'$, we have $|\calV'| > |\calV''|$ and so the probability that the largest label is assigned to a vertex in $\calV'$ is greater than $1/2$. The distribution of the order statistics of $L$ restricted to $\calV'$ is independent of this event, and the same holds for $\calV''$.
These observations and \eqref{m21.1} imply that
\begin{align*}
\frac{\P(\calN(\calG,2) \cap\{  K_1=v_*, K_2 = y\})}{\P(\calN(\calG,1) \cap\{  K_1=v_*\})}
\geq
\frac{ \frac 1 2 \prod _{x\in \calV'} (1/ \n_{v_*}^{\calG'}(x)
\prod _{x\in \calV''} (1 / \n_{y}^{\calG''}(x))}
{ \prod _{x\in \calV} (1 / \n^\calG_{v_*}(x))}.
\end{align*}
Note that the factors corresponding to $x\notin \Gamma$ are identical in the numerator and denominator, so
\begin{align}\label{y31.1}
\frac{\P(\calN(\calG,2) \cap\{  K_1=v_*, K_2 = y\})}{\P(\calN(\calG,1) \cap\{  K_1=v_*\})}
\geq
\frac 1 2 
\prod_{j=2}^m
\frac{ \n^\calG_{v_*}(z_j) }{ \n_{v_*}^{\calG'}(z_j)}
\prod_{j=0}^1
\frac{ \n^\calG_{v_*}(z_j) }{ \n_{y}^{\calG''}(z_j)}.
\end{align}
Since $m\leq k$,
\begin{align}\nonumber
\frac{ \n^\calG_{v_*}(z_0) }{ \n_{y}^{\calG''}(z_0)}
&=
 \n^\calG_{v_*}(y) ( \n_{y}^{\calG''}(y))^{-1}
= \frac{ d(d^{k-m+1} -1)}{d-1}
\left(\frac{d( d^{k-m+2} -1)}{d-1}\right)^{-1}\\
& = d^{-1}( 1- d^{-(k-m+1)}) (1 - d^{-(k-m+2)})^{-1} 
\geq d^{-1}( 1- d^{-(k-m+1)})\geq \frac {d-1}{d^2}.\label{y31.2}
\end{align}
The set of vertices corresponding to $ \n^\calG_{v_*}(z_1) $ contains the set corresponding to $ \n_{y}^{\calG''}(z_1)$ so we have the following 
bound for the second factor in \eqref{y31.1} corresponding to $\calV''$,
\begin{align}\label{y31.3}
\frac{ \n^\calG_{v_*}(z_1) }{ \n_{y}^{\calG''}(z_1)}
\geq 1.
\end{align}

Since $\calV' \subset \calV$,  $\n^\calG_{v_*}(z_j) \geq \n_{v_*}^{\calG'}(z_j)$, for all $j$.
This implies that
\begin{align}\label{y31.4}
\prod_{j=2}^m
\frac{ \n^\calG_{v_*}(z_j) }{ \n_{v_*}^{\calG'}(z_j)}
\geq 1.
\end{align}
We combine \eqref{y31.1}-\eqref{y31.4} to obtain for $d\geq 3$,
\begin{align}\label{y31.5}
\frac{\P(\calN(\calG,2) \cap\{  K_1=v_*, K_2 = y\})}{\P(\calN(\calG,1) \cap\{  K_1=v_*\})}
\geq
\frac 1 2 
\frac {d-1}{d^2}
 \geq \frac 1 {4 d} .
\end{align}

We will now derive an upper bound.
For $1 \leq j \leq m-2$,
let $\calV^j_1$ and $\calV^j_2$ be obtained by removing the edge between $z_j$ and $z_{j+1}$ from $\tree$; we assume that
 $v_*\in\calV^j_1$. The corresponding graphs will be called $\calG^j_1$ and $\calG^j_2$.
Recall the following notation, $A(\calV^j_1, \calV^j_2) = \calN(\calG^j_1,1) \cap \calN(\calG^j_2,1)$.
Let $K^{j,i}_1$ be the location of the highest peak in $\calG^j_i$.
It follows from the proof of Lemma \ref{j24.1} that 
\begin{align*}
\calN(\calG,2) \cap \{K_1=v_*, K_2 = y\} 
= \bigcup_{j=1}^{m-2} (A(\calV^j_1, \calV^j_2) \cap \{K^{j,1}_1=v_*, K^{j,2}_1  = y\} ).
\end{align*}
These observations and \eqref{m21.1} imply that
\begin{align*}
\frac{\P(\calN(\calG,2) \cap\{  K_1=v_*, K_2 = y\})}{\P(\calN(\calG,1) \cap\{  K_1=v_*\})}
\leq
\frac{\sum_{i=1}^{m-2}\left( \prod _{x\in \calV^i_1}( 1 / \n_{v_*}^{\calG^i_1}(x))
\prod _{x\in \calV^i_2} (1 / \n_{y}^{\calG^i_2}(x)) \right)}
{ \prod _{x\in \calV} (1 / \n^\calG_{v_*}(x))}.
\end{align*}

The factors corresponding to $x\notin \Gamma$ are identical in the numerator and denominator, so
\begin{align}\label{y31.6}
\frac{\P(\calN(\calG,2) \cap\{  K_1=v_*, K_2 = y\})}{\P(\calN(\calG,1) \cap\{  K_1=v_*\})}
\leq
\sum_{i=1}^{m-2}
\prod_{j=i+1}^m
\frac{ \n^\calG_{v_*}(z_j) }{ \n_{v_*}^{\calG^i_1}(z_j)}
\prod_{j=0}^i
\frac{ \n^\calG_{v_*}(z_j) }{ \n_{y}^{\calG^i_2}(z_j)}.
\end{align}
For $m\leq k$,  $d\geq 3$ and $i\geq 1$,
\begin{align*}
\frac{ \n^\calG_{v_*}(z_0) }{ \n_{y}^{\calG^i_2}(z_0)}
&
= \frac{d( d^{k-m+1} -1)}{d-1}
\left(\frac{d( d^{k-m+1+i} -1)}{d-1}\right)^{-1}\\
& = d^{-i}( 1- d^{-(k-m+1)}) (1 - d^{-(k-m+1+i)})^{-1} \\
&\leq d^{-i}( 1- d^{-(k-m+1)}) (1 +2  d^{-(k-m+1+i)})
\leq 3 d^{-i}.
\end{align*}
We estimate the other factors corresponding to $\calV^i_2$ as follows,
for $1\leq j \leq i$,
\begin{align*}
\frac{ \n^\calG_{v_*}(z_j) }{ \n_{y}^{\calG^i_2}(z_j)}
&= \frac{d( d^{k-m+1+j} -1)}{d-1}
\left(\frac{d( d^{k-m+1+i} -1)}{d-1} - \frac{d( d^{k-m+j} -1)}{d-1}\right)^{-1}\\
& = d^{j-i}( 1- d^{-(k-m+1+j)}) 
(1 - d^{j-i-1})^{-1} 
\leq d^{j-i} (1 +2 d^{j-i-1})
\leq 3 d^{j-i}.\nonumber
\end{align*}
Hence
\begin{align}\label{y31.8}
\prod_{j=0}^i \frac{ \n^\calG_{v_*}(z_j) }{ \n_{y}^{\calG^i_2}(z_j)}
\leq \prod_{j=0}^i 3 d^{j-i} = 3^{i+1} d^{- i(i+1)/2}.
\end{align}

Next we deal with factors corresponding to $\calV^i_1$, for $i+1 \leq j \leq m$,
\begin{align*}
\frac{ \n^\calG_{v_*}(z_j) }{ \n_{v_*}^{\calG^i_1}(z_j)}
&
= \frac{d( d^{k-m+1+j} -1)}{d-1}
\left(\frac{d( d^{k-m+1+j} -1)}{d-1} - \frac{ d(d^{k-m+1+i} -1)}{d-1}\right)^{-1}\\
& = ( 1- d^{-(k-m+1+j)}) (1 - d^{i-j})^{-1} 
\leq  1+2 d^{i-j} .
\end{align*}
This implies that
\begin{align}\label{y31.9}
\prod_{j=i+1}^m
\frac{ \n^\calG_{v_*}(z_j) }{ \n_{v_*}^{\calG^i_1}(z_j)}
&\leq 
\prod_{j=i+1}^m (1+2 d^{i-j})
= \exp \left( \sum _{j=i+1}^m \log(1+2 d^{i-j})\right)
\leq \exp \left( \sum _{j=i+1}^m 2 d^{i-j}\right)\\
&\leq \exp \left( \sum _{j=i+1}^\infty 2 d^{i-j}\right)
= \exp \left( \frac 2 {d-1}\right).
\nonumber
\end{align}
We combine \eqref{y31.6}-\eqref{y31.9} to obtain for $d\geq 3$,
\begin{align}\label{y31.10}
\frac{\P(\calN(\calG,2) \cap\{  K_1=v_*, K_2 = y\})}{\P(\calN(\calG,1) \cap\{  K_1=v_*\})}
&\leq
\sum_{i=1}^{m-2} \exp \left( \frac 2 {d-1}\right) 
3^{i+1} d^{- i(i+1)/2}\\
&\leq
\sum_{i=1}^{\infty} \exp \left( \frac 2 {d-1}\right) 
3^{i+1} d^{- i(i+1)/2} < \infty.\nonumber
\end{align}
Since the lower and upper bounds in \eqref{y31.5} and \eqref{y31.10} do not depend
on $m = \dist(y, v_*)$, part (c) of the theorem follows.

\medskip
(d) 
The number of $y\in \calV$ with $\dist(y, v_*) \geq 2$ is equal 
to $(d^k-d) (d+1)/(d-1)$. Since 
\begin{align*}
d^k/2 \leq (d^k-d) (d+1)/(d-1) \leq 2 d^k,
\end{align*}
for $d\geq 3$ and $k\geq 2$, part (c) of the lemma implies that for every $d\geq 3$ there exists $c_2$ such that for all $k\geq 2$ and $y\in \calV$ with $\dist(y, v_*) \geq 2$, we have
\begin{align}\label{j1.1}
d^{-k}/c_2 \leq \P_2( K_1=v_*, K_2 = y) \leq c_2 d^{-k}. 
\end{align}

Consider $y_1,y_2\in \calV$ and let $m_i = \dist(y_i, v_*)$ for $i=1,2$. Let $m=\min(m_1,m_2)$. We use \eqref{y30.1} and \eqref{j1.1} to see that
\begin{align*}
&\P(\calN(\calG,2) \cap (\{K_1=y_1, K_2 = y_2\} \cup \{K_1=y_2, K_2 = y_1\})) \\
&\leq 4 (d-1)^{-m(m+1)/2}
\P(\calN(\calG,2) \cap( \{K_1=v_*, K_2 = y_2\} \cup \{K_1=v_*, K_2 = y_1\}))\nonumber \\
&= 4 (d-1)^{-m(m+1)/2}
\P_2( \{K_1=v_*, K_2 = y_2\} \cup \{K_1=v_*, K_2 = y_1\})
\P(\calN(\calG,2) 
\nonumber \\
&\leq 8 (d-1)^{-m(m+1)/2} c_2 d^{-k}\P(\calN(\calG,2) .\nonumber
\end{align*}
This implies that for some $c_3$ depending on $d$ but not on $k$,
\begin{align}\label{j1.3}
&\P_2( K_1 = y_1, \dist(K_2, v_*) \leq m_1) =
\sum_{j=0}^{m_1}
\sum_{\substack{\{y_2:\  \dist(y_2, v_*) =j,\\  \dist (y_1,y_2) \geq 2\} }}
\P_2( K_1=y_1, K_2 = y_2)\\
& \leq \sum_{j=0}^{m_1}
(d+1) d^{j-1}
8 (d-1)^{-j(j+1)/2} c_2 d^{-k}\leq c_3 d^{-k}.\nonumber
\end{align}
A similar argument yields
\begin{align}\label{j1.4}
\P_2(  K_2 = y_2, \dist(K_1, v_*) \leq m_2) 
\leq c_3 d^{-k}.
\end{align}
We use \eqref{j1.3} and \eqref{j1.4} to obtain
\begin{align}\label{j1.5}
\P_2&\left(  \max_{i=1,2} \dist(K_i, v_*) \leq k-n/2 \right)\\
&=
\sum_{j=0} ^{k-n/2}
\sum_{ \{y_1:\ \dist(y_1, v_*) = j\} }
\P_2( K_1 = y_1, \dist(K_2, v_*) \leq m_1) \nonumber \\
&\quad +
\sum_{j=0} ^{k-n/2}
\sum_{ \{y_2:\ \dist(y_2, v_*) = j\} }
\P_2( K_2 = y_2, \dist(K_1, v_*) \leq m_2) \nonumber \\
& \leq \sum_{j=0} ^{k-n/2}
\sum_{ \{y:\ \dist(y, v_*) = j\} }
2 c_3 d^{-k}
\leq \left( 1 +\sum_{j=1} ^{k-n/2}
(d+1) d^{j-1}\right)
2 c_3 d^{-k} \nonumber \\
& = \left( 1 + (d+1) \frac{ d^{k-n/2} -1 }{d-1} \right)  2 c_3 d^{-k}
\leq c_4 d^{-n/2}.  \nonumber  
\end{align}

Suppose that $|\dist(K_1, K_2) - k | \geq n$. It is possible that 
$\min_{i=1,2} \dist(K_i, v_*) \geq n/2$. 
Suppose that the last condition is not satisfied and consider the case when 
$\dist(K_1, v_*) < n/2$. If $\dist(K_2, v_*) > k-n/2$ then, by the triangle inequality, $\dist(K_1, K_2) \geq k - n$, contradicting the assumption that 
$|\dist(K_1, K_2) - k | \geq n$. It follows that 
$\dist(K_2, v_*) \leq k-n/2$.
Similarly, if  
$\dist(K_2, v_*) < n/2$ then 
$\dist(K_1, v_*) \leq k-n/2$.
This argument proves that
\begin{align*}
&\{ |\dist(K_1, K_2) - k | \geq n\} \\
&\quad \subset \left\{ \min_{i=1,2} \dist(K_i, v_*) \geq n/2\right \}
\cup \left\{ \max_{i=1,2} \dist(K_i, v_*) \leq k-n/2\right\}.
\end{align*}
It follows from the last formula, \eqref{j1.5} and part (b) of the theorem that
 for every $d\geq 3$ there exist $c_4$ and $c_5$ such that for $k\geq 2$ and $n\geq 2 \sqrt{8k} $,
\begin{align*}
\P_2&\left(|\dist(K_1, K_2) - k | \geq n\right) 
\leq 8 \exp(- (n/4) \log (d-1))
+ c_4 d^{-n/2}\\
& \leq  8 \exp(- (n/4) \log (d-1))
+ c_4 \exp(-(n/2)\log d)\\
& \leq  c_5 \exp(- (\sqrt{k}/64) \log d).
\end{align*}
\end{proof}

\section{Acknowledgments}

We are grateful to 
Omer Angel, J\'er\'emie Bettinelli, Sara Billey, Shuntao Chen, Ivan Corwin,
Ted Cox, Nicolas Curien, Michael Damron,  Dmitri Drusvyatskiy, Rick Durrett, Martin Hairer, Christopher Hoffman, Lerna Pehlivan,
Doug Rizzolo, Bruce Sagan, Timo Seppalainen, Alexandre Stauffer, Wendelin Werner and Brent Werness
for the most helpful advice.
The first author is grateful to the Isaac Newton Institute for Mathematical Sciences, where this research was partly carried out, for the hospitality and support.

\bibliographystyle{alpha}
\bibliography{isolated}

\end{document}